\documentclass[12pt]{amsart}
\pdfoutput=1
\usepackage[margin=1in]{geometry}
\usepackage{amsmath, amsthm, amsfonts, amssymb, mathrsfs, mathdots}
\usepackage{microtype}
\usepackage[inline]{enumitem}
\usepackage{fancyvrb, xstring}

\usepackage{hyperref}
\usepackage[dvipsnames]{xcolor}
\hypersetup{%
  bookmarks,
  pagebackref = true,
  destlabel = true,
  hypertexnames = false,
  colorlinks,
  linkcolor = Purple,
  citecolor = Blue,
  urlcolor = Blue,
}

\usepackage[mathscr]{euscript}

\usepackage{tikz, tikz-cd, ytableau, graphicx}
\usetikzlibrary{positioning,graphs,patterns}
\ytableausetup{centertableaux, boxsize=1em}

\numberwithin{equation}{subsection}
\newtheorem{theorem}[equation]{Theorem}

\newtheorem{proposition}[equation]{Proposition}
\newtheorem{lemma}[equation]{Lemma}
\newtheorem{corollary}[equation]{Corollary}
\newtheorem{conjecture}[equation]{Conjecture}

\theoremstyle{definition}
\newtheorem{rmk}[equation]{Remark}
\newenvironment{remark}[1][]{\begin{rmk}[#1] \pushQED{\qed}}{\popQED \end{rmk}}
\newtheorem{eg}[equation]{Example}
\newenvironment{example}[1][]{\begin{eg}[#1] \pushQED{\qed}}{\popQED \end{eg}}
\newtheorem{defn}[equation]{Definition}

\makeatletter
\@addtoreset{equation}{section}
\renewcommand{\thesubsection}{%
  \ifnum\c@subsection<1 \@arabic\c@section
  \else \thesection.\@arabic\c@subsection
  \fi
}
\makeatother

\makeatletter
\newcommand{\extp}{\@ifnextchar^\@extp{\@extp^{\,}}}
\def\@extp^#1{\mathop{\bigwedge\nolimits^{\!#1}}}
\makeatother

\newcommand{\bA}{\mathbf{A}}

\newcommand{\bC}{\mathbf{C}}

\newcommand{\sE}{\mathscr{E}}
\newcommand{\bF}{\mathbf{F}}

\newcommand{\bG}{\mathbf{G}}

\newcommand{\sG}{\mathscr{G}}

\newcommand{\rH}{\mathrm{H}}

\newcommand{\bK}{\mathbf{K}}

\newcommand{\bL}{\mathbf{L}}

\newcommand{\sO}{\mathscr{O}}
\newcommand{\bP}{\mathbf{P}}

\newcommand{\sQ}{\mathscr{Q}}

\newcommand{\rR}{\mathrm{R}}
\newcommand{\sR}{\mathscr{R}}
\newcommand{\bS}{\mathbf{S}}

\newcommand{\fS}{\mathfrak{S}}

\newcommand{\sS}{\mathscr{S}}

\newcommand{\sT}{\mathscr{T}}

\newcommand{\sU}{\mathscr{U}}

\newcommand{\sV}{\mathscr{V}}

\newcommand{\sW}{\mathscr{W}}

\newcommand{\bZ}{\mathbf{Z}}

\renewcommand{\phi}{\varphi}
\renewcommand{\emptyset}{\varnothing}

\newcommand{\injects}{\hookrightarrow}
\newcommand{\surjects}{\twoheadrightarrow}
\renewcommand{\tilde}[1]{\widetilde{#1}}
\newcommand{\ol}[1]{\overline{#1}}

\newcommand{\stacks}[1]{%
  \cite%
  [Tag {\href{https://stacks.math.columbia.edu/tag/#1}{\textsc{#1}}}]%
  {stacks}
}

\makeatletter
\def\Ddots{\mathinner{\mkern1mu\raise\p@
\vbox{\kern7\p@\hbox{.}}\mkern2mu
\raise4\p@\hbox{.}\mkern2mu\raise7\p@\hbox{.}\mkern1mu}}
\makeatother

\DeclareMathOperator{\codim}{codim}

\DeclareMathOperator{\rad}{rad}

\DeclareMathOperator{\rank}{rank}

\DeclareMathOperator{\End}{End}

\DeclareMathOperator{\Sym}{Sym}

\DeclareMathOperator{\Tor}{Tor}

\DeclareMathOperator{\Spec}{Spec}

\DeclareMathOperator{\res}{res}

\DeclareMathOperator{\gr}{gr}

\newcommand{\GL}{\mathbf{GL}}

\newcommand{\Sp}{\mathbf{Sp}}

\newcommand{\Gr}{\mathbf{Gr}}
\newcommand{\Flag}{\mathbf{Flag}}
\newcommand{\IFlag}{\mathbf{IFlag}}
\newcommand{\IGr}{\mathbf{IGr}}
\newcommand{\CoIGr}{\mathbf{CoIGr}}
\newcommand{\OGr}{\mathbf{OGr}}

\newcommand\inlinedef[1]{\textbf{#1}}

\def\R{\mathscr{R}} % tautological sub bundle
\def\Q{\mathscr{Q}} % tautological quotient bundle
\def\O{\mathscr{O}} % structure sheaf

\newcommand{\gl}{\mathfrak{gl}}   % general linear Lie algebra
\renewcommand{\sp}{\mathfrak{sp}} % symplectic Lie algebra
\newcommand{\so}{\mathfrak{so}}   % orthogonal Lie algebra
\newcommand{\LR}[2]{\ensuremath{c^{#1}_{#2}}}  % Littlewood--Richardson coefficients

\newcommand\LWA[3]{(#1; #2)(-#3)}

\newcommand\WA[2]{\LWA{0}{#1}{#2}}

\newcommand\Hxi[3]{\rH^{#1}\wedge^{#2}(#3)}

\DeclareMathOperator{\Span}{span}
\DeclareMathOperator{\Inv}{inv}
\DeclareMathOperator{\Neg}{neg}
\DeclareMathOperator{\Nsp}{nsp}

\DeclareMathOperator{\nnormal}{nonnormal}
\DeclareMathOperator{\sing}{sing}

\DeclareMathOperator{\Kal}{\mathbf{Kal}}       % Kalman variety
\DeclareMathOperator{\IKal}{\mathbf{IKal}}     % Isotropic Kalman variety
\DeclareMathOperator{\CoIKal}{\mathbf{CoIKal}} % Coisotropic Kalman variety
\DeclareMathOperator{\OKal}{\mathbf{OKal}}     % Orthogonal Kalman variety

\begin{document}
\title{Syzygies of Isotropic Kalman Varieties}
\author{Suhas Vadan Gondi} \email[Suhas Vadan Gondi]{sgondi@ucsd.edu}
\author{Sarah Kumar}       \email[Sarah Kumar]{sak018@ucsd.edu}
\author{Abhik Pal}         \email[Abhik Pal]{apal@ucsd.edu}
\address{%
  Department of Mathematics \\
  University of California San Diego \\
  La Jolla CA 92093 \\
  USA}

\begin{abstract}
  Let $L$ be a subspace of a complex vector space $V$ and fix $s \leq \dim{L}$.
  The (type A) Kalman variety consists of all endomorphisms of $V$ that have an $s$-dimensional invariant subspace in $L$.
  We introduce a generalization where $V$ and $L$ are symplectic vector spaces.
  We fix an isotropic subspace $W \subseteq V$ satisfying $W^\perp = W \oplus L$.
  The isotropic (type C) Kalman variety consists of symplectic morphisms of $V$ that have an invariant coisotropic subspace of a prescribed dimension inside $W^\perp$.
  We are mainly interested in studying the Lagrangian case.
  In type C, we prove analogues of results known for type A Kalman varieties; in particular, we determine the defining equations, compute geometric invariants, and analyze their singularities.
  We conjecture the existence of a long exact sequence relating the structure sheaves.
  Based on the results in the symplectic case, we describe Kalman variety analogues with respect to endomorphisms of odd orthogonal (type B) and even orthogonal (type D) vector spaces.
\end{abstract}
\maketitle

\section{Introduction}
A classical problem in linear algebra is to understand the invariant subspaces of an endomorphism of a vector space.
If $V$ is a complex $n$-dimensional vector space and $L \subseteq V$ a $d$-dimensional subspace then the space of all endomorphisms with an $s$-dimensional invariant subspace in $L$ is the algebaic subset
\begin{equation*}
  \Kal(s, d, n) = \left\{ \varphi \in \End(V) \;\vert\; \exists U \subseteq L, \dim{U} = s, \varphi(U) \subseteq U \right\}
\end{equation*}
of $\End(V)$ known as the (type A) Kalman variety.
The study of the geometric properties of $\Kal(s, d, n)$ in \cite{os2012:matrices} was motivated by Kalman's observability criteria in control theory \cite{kal1960:optimalcontrol}; recent work on Kalman varieties has focused on determining their minimal generating equations \cite{sam2012:equations,hua2020:equations} and studying generalizations \cite{os2022:tensors,ssv2023:degrees,ssw2025:nonlinear}.
The main goal of this paper is to generalize Kalman varieties when $V$ and $L$ are symplectic vector spaces.

Let $V$ be a symplectic vector space of dimension $2(n+r)$, $W \subseteq V$ an $r$-dimensional isotropic subspace, and $L \subseteq V$ a $2n$-dimensional symplectic subspace satisfying $L \oplus W = W^\perp$.
The isotropic Kalman variety $\IKal(m, r, n) \subseteq \sp(V)$ is defined by
\begin{align*}
  \IKal(m, r, n)
  &= \{ \varphi \in \sp(V) \;|\; \exists\, \widetilde U \subseteq W^\perp \text{ coisotropic}, \dim{\widetilde U} = 2n - m + r, \varphi(\widetilde U) \subseteq \widetilde U \} \\
  &= \{ \varphi \in \sp(V)\;|\; \exists\, U \supseteq W \text{ isotropic}, \dim{U} = m+r, \varphi(U) \subseteq U \}.
\end{align*}
The equivalence of the two stated definitions is obtained by taking symplectic complements.
The isotropic Kalman variety thus consists of symplectic endomorphisms with an invariant coisotropic subspace $\widetilde U$ in a fixed coisotropic subspace $W^\perp$.
The space of endomorphisms with invariant \emph{isotropic} subspaces inside a fixed \emph{isotropic} subspace also defines an algebraic set in $\sp(V)$; our methods are not directly applicable in that setup and we do not study this variant here.
Throughout this paper, we fix vector spaces $V$, $W$, and $L$ and the parameters $r$ and $n$; we denote $\IKal(m, r, n)$ by $\IKal(m)$.

In this paper, we establish analogues of two important results known in the type A case.
Following Ottaviani--Sturmfels \cite{os2012:matrices}, we determine defining equations for isotropic Kalman varieties, and following Sam \cite{sam2012:equations} and Huang \cite{hua2020:equations}, we show that singular loci of isotropic Kalman varieties exhibit the unexpected ``nesting'' behavior known in type A.

In \cite{os2012:matrices}, Ottaviani--Sturmfels show that the equations that set theoretically define $\Kal(s, d, n)$ are obtained by taking the $(d - s + 1) \times (d - s + 1)$ minors of the reduced Kalman matrix $\Theta^\bA(\varphi) = \begin{bmatrix} C & CA & \dots & CA^{d - 1} \end{bmatrix}^\intercal$ constructed by writing a matrix $\varphi \in \End(V)$ as a block matrix $\varphi = \left[\begin{smallmatrix} A & B \\ C & D \end{smallmatrix}\right]$ with respect to the decomposition $V=L\oplus V/L$.
In the symplectic case, we first write a generic symplectic endomorphism $\varphi \in \sp(V)$ as a block matrix
\begin{equation*}
  \varphi =
  \begin{bmatrix}
    {A_0} & {A_1} & {B_0} \\
    {A_2} & {A_3} & {B_1} \\
    {C_0} & {C_1} & {D_0}
  \end{bmatrix}
\end{equation*}
with respect to the decomposition $V = L \oplus W \oplus W^\ast$ and then define the type C reduced Kalman matrix
\begin{equation*}
  \Theta
  = \Theta^{\bC}(\varphi)
  = \begin{bmatrix}
    C_0 & C_0 A_0 & \dots & C_0 A_0^{2n - 1}
  \end{bmatrix}^\intercal
\end{equation*}
associated to $\varphi$.
We prove an analogue of Ottaviani--Sturmfels's result.
\begin{theorem}[Theorem~\ref{t:defn-eq}]
  The variety $\IKal(m)$ is set theoretically defined by the vanishing of the $(m + 1) \times (m + 1)$ minors of $\Theta^{\bC}(\varphi)$ along with the equations $C_1=0$ and $C_0 A_0^k A_1=0$ for all $0 \leq k \leq 2n - 1$.
\end{theorem}

In \cite{sam2012:equations} and \cite{hua2020:equations}, Sam and Huang respectively improve on Ottaviani--Sturmfels's result and obtain the radical ideal generators of Kalman varieties.

Sam and Huang also discovered an unexpected relationship between the singular and nonnormal loci of Kalman varieties and showed that $\nnormal(\Kal(s, d, n)) = \Kal(s + 1, d, n)$.
We prove that the analogue of this holds in the isotropic case.

\begin{theorem}[Theorem~\ref{t:nonnormal}]
  The nonnormal locus of $\IKal(m)$ coincides with the singular locus and we have
  \begin{equation*}
    \nnormal(\IKal(m))
    = \sing(\IKal(m))
    = \IKal(m - 1).
  \end{equation*}
\end{theorem}

Using the two theorems above, we compute the minimal generators for the Lagrangian Kalman variety $\IKal(1)$ when $\dim{L} = 2$.
Proving this in general is not so straightforward.
However, based on the theorems above and \cite[Theorem~3.1]{hua2020:equations}, we conjecture the existence of a long exact sequence relating the free resolutions of arbitrary isotropic Kalman varieties and their normalizations.
Since the case $m = n$ corresponds to a Lagrangian Kalman variety, one consequence of our conjecture would be a complete description of all syzygies ($\Tor$-groups) of $\IKal(n)$.

\begin{conjecture}[{Conjecture~\ref{conj:ikal-les}}]
  Let $\sO_m$ denote the structure sheaf of $\IKal(m)$ and $\widetilde\sO_m$ that of its normalization.
  Then there exists a long exact sequence
  \begin{equation*}
    0 \to
    \O_n \to
    \widetilde\O_n \to
    \widetilde\O_{n - 1}(-1) \to
    \dots \to
    \widetilde\O_m(-(n - m)^2) \to
    \dots \to
    \widetilde\O_0(- n^2) \to
    0
  \end{equation*}
  of structure sheaves.
\end{conjecture}

The main techniques we use involve analyzing the syzygies of varieties in terms of cohomology of vector bundles.
In general, these techniques are known to work very well when the varieties are normal \cite[Theorem~5.1.3]{Wey:CohomVect}; in the study of nilpotent orbit closures of type A, for example, the orbit closures are normal and these methods give complete information about their syzygies \cite[Chapter~8]{Wey:CohomVect}.
For every simple Lie algebra not of type A, there always exists a nonnormal orbit closure and very little is known about their equations and free resolutions.
Since isotropic Kalman varieties are not normal in general, they form an important new class of examples for studying this technique in type C.

\medskip

We now summarize our main approach to generalizing the definition to the symplectic case.
Compta, Helmke, Pe\~{n}a, and Puerta observe in \cite{chpp2006:simultaneous} that the space of endormorphisms and their invariant subspaces can be identified with vector bundles over a Grassmannian.
In particular, we may define Kalman varieties purely in terms of vector bundles over Grassmannians.
Let $X = \Gr(s, L)$ denote the Grassmannian of $s$-dimensional subspaces of $L$.
Recall that the varieties
\begin{align}
  \label{def:R}
  \sR &= \{ (v, R) \;|\; R \in \Gr(s, L), v \in R \} \subseteq L \times X,  \\
  \label{def:Q}
  \sQ &= \{ (v, R) \;|\; R \in \Gr(s, L), v \in L / R \} = (L \times X) / \sR
\end{align}
along with the projection morphisms to $X$ are respectively called the tautological sub- and quotient bundles on $X$.
In terms of these bundles, the cotangent bundle $\Omega$ is isomorphic to $\sR \otimes \sQ^\ast$.
If $L \subseteq V$, then there is a natural inclusion $X \hookrightarrow \Gr(s, V)$ and the cotangent bundle $\Omega_0 \cong \sR_0 \otimes \sQ_0^*$  of $\Gr(s, V)$ restricts to
\begin{equation*}
  \xi = \sR \otimes (\sQ^\ast  \oplus (V / L)^\ast) = \Omega \oplus \sR \otimes (V / L)^\ast
\end{equation*}
on $X$.
Let $\sO_X$ denote the structure sheaf of $X$ and note that $\xi$ is a subbundle of $\sO_X \otimes \End(V)$.
If $\eta = \sO_X \otimes \End(V) / \xi$ is the quotient and $Z$ is the total space of $\eta^\ast$, then there is a diagram
\begin{equation*}
  \begin{tikzcd}
    Z \ar["\subseteq", hook]{r} \ar["\pi' = \pi\vert_Z"]{d} & \Gr(s, L) \times \End(V) \ar["\pi_1"]{r} \ar["\pi = \pi_2"]{d} & \Gr(s, L) \\
    Y\ar["\subseteq", hook]{r} & \End(V) &
  \end{tikzcd}
\end{equation*}
where the $\pi_i$ denote the projection morphisms.
When the projection $\pi$ is restricted to $Z$, the Kalman variety naturally arises as the subvariety $Y = \pi(Z) = \Kal(s, L, V)$  of $\End(V)$.
As an $\sO_{X \times \End(V)}$-module, the structure sheaf $\sO_Z$ of $Z$ is resolved by the Koszul complex
\begin{equation*}
  \bK(\xi)_\bullet \colon
  \quad
  0 \to
  \bigwedge^t (\pi_1^\ast \xi)
  \to \dots
  \to \bigwedge^2 (\pi_1^\ast \xi)
  \to \pi_1^\ast\xi
  \to \O_{X \times \End(V)}
\end{equation*}
on $\xi$. 
By cohomology and base change, the pushforward of $\bK(\xi)_\bullet$ along $\pi$ can be identified with cohomology groups $\rH^\bullet(X, \bK(\xi)_\bullet)$.
When $\pi'$ is birational, these groups in fact determine a free resolution of the normalization of $\Kal(s, d, n)$ as a $\Sym(\End(V)^\ast)$-module.

If $L$ is a symplectic vector space of dimension $2n$ and $m \leq n$, then the space of all $m$-dimensional isotropic subspaces of $L$ is a subvariety $\IGr(m, L)$ of $\Gr(m, L)$ called the isotropic Grassmannian.
If $V$ is a symplectic vector space of dimension $2(n + r)$ containing $L$ and $W \subseteq V$ is a fixed isotropic subspace satisfying $W^\perp=W\oplus L$ and $\dim{W} = r$, then there is the embedding $\IGr(m, L) \hookrightarrow \IGr(m + r, V)$ of isotropic Grassmannians obtained by $R \mapsto R \oplus W$.
The main class of examples we will consider is the case where the invariant subspaces are Lagrangian.
In this case, $m = n$ and the cotangent bundle on $\IGr(n + r, V)$ restricts to $\xi = \Sym^2(\sR \oplus \sW)$ on $\IGr(n, L)$.
Carrying out the construction above gives us the expected set theoretic description
\begin{equation*}
  \IKal(n) = \left\{ \varphi \in \sp(V) \;|\; \exists U \in \IGr(n, L), \varphi(U \oplus W) \subseteq U \oplus W \right\}.
\end{equation*}
of the Lagrangian Kalman variety.
Studying singular loci in this case naturally leads to the general class of isotropic Kalman varieties.
However, for a general isotropic Kalman variety, the bundle $\xi$ fails to be semisimple.
When $m < n$, the cohomology calculations and the combinatorics involved get quite complicated; we describe this construction in detail in \S\ref{sec:restriction-defn}.

\medskip

Historically, the rank condition on $\Theta^{\bA}(\varphi)$ was studied in the context of control theory where it is known as Kalman's observability condition \cite{kal1960:optimalcontrol}; it also appears in the study of certain families of PDEs \cite{sk1985:systems,csz2024:large}.
Algebras related to the Kalman variety were known to algebraic geometers as early as the 1980s: irreducibility and dimension calculations when $\Theta^\bA(\varphi)$ has maximal rank appear in Helmke's PhD thesis \cite{hel1982:zur} and an algebraic proof of the rank condition when $s = 1$ appear in \cite{she1984:common}.
Osetinskiĭ, Vasil'ev, and Vainsteĭn studied the invariant theory of the $\GL(V)$-action on algebras related to Kalman varieties in \cite{ovv2006:geometric} and
Compta, Helmke, Pe\~{n}a, and Puerta realized the space of endomorphisms and their invariant subspaces as vector bundles over Grassmannians in \cite{chpp2006:simultaneous}.

In \cite{os2012:matrices}, Ottaviani and Sturmfels first carried out a systematic study of the algebraic geometry of Kalman varieties and posed the general problem of calculating the prime ideal generators.
In  \cite{sam2012:equations}, Sam discovered a long exact sequence relating the coordinate rings of various Kalman varieties and their normalizations; using this long exact sequence he was able to determine the prime ideal generators when $\dim{L} = 3$.
Huang showed in \cite{hua2020:equations} that the exact sequence exists for all values of $\dim{L}$ and was able to determine the prime ideal generators and syzygies of the Kalman variety.

Analogues of Kalman varieties for higher order tensors have been studied by Ottaviani--Shahidi and  Shahidi--Sodomaco--Ventura in  \cite{os2022:tensors} and \cite{ssv2023:degrees}.
Salizzoni, Sodomaco, and Weigert introduced non-linear generalization of Kalman varieties in \cite{ssw2025:nonlinear}; they fix an irreducible projective variety $X \subseteq \bP(V)$ and consider the subset of $\End(V)$ with invariant subspaces in $X$. This paper presents yet another variation on these themes.

\subsection*{Outline}
\begin{enumerate}
\item[\S\ref{sec:prelim}]
  We summarize the preliminary material used in the rest of the paper.
  This section contains no new results.
  We review  the relevant representation theoretic formulae for Schur functors and present an overview of Weyman's geometric method for calculating syzygies.
  We also state the two important special cases (for isotropic Grassmannians and relative flag varieties) of the Borel--Weil--Bott theorem used throughout this paper.

\item[\S\ref{sec:restriction-defn}]
  We describe a construction of Kalman varieties using vector bundles over Grassmannians.
  This description is then used to define the isotropic Kalman variety $\IKal(m)$.

\item[\S\ref{sec:geom-props}]
  This section focuses on proving numerous geometric properties of isotropic Kalman varieties.
  We calculate the defining equations and show that isotropic Kalman varieties are defined by minors of a symplectic analogue of the reduced Kalman matrix (Theorem~\ref{t:defn-eq}).
  One of the main results in this section is the identification of the singular and nonnormal loci of $\IKal(m)$ with $\IKal(m - 1)$ (Theorem~\ref{t:nonnormal}).
  We show that normalizations of isotropic Kalman varieties have rational singularities (Theorem~\ref{t:norm-ikal-rational-sing}), describe a desingularization (Proposition~\ref{p:birational}), and calculate invariants such as dimension and codimension (Proposition~\ref{p:codim}).
  Using results on rational singularities and the desingularization, we construct free resolutions (Proposition~\ref{p:norm-ikal-res}).

\item[\S\ref{sec:Lagrangian}]
  We use the tools developed in previous sections to study Lagrangian Kalman varieties.
  We compute the geometric invariants and the defining equations in this special case (Corollary~\ref{c:defn-eq:lagrangian}).
  Our main result in this section is the calculation of the minimal generators for Lagrangian Kalman varieties when $\dim{L} = 2$ (Theorem~\ref{t:dim-2-min-eq}).

\item[\S\ref{sec:exact-sequence}]
  We state out main conjecture on the existence of a long exact sequence involving structure sheaves of normalizations of isotropic Kalman varieties (Conjecture~\ref{conj:ikal-les}).
  We present general heuristics for proving such a result and state some results in support of the conjecture.

\item[\S\ref{sec:ortho}]
  We briefly describe how the results of the rest of paper apply in the odd and even orthogonal case.
  This section only summarizes the final results as all proofs can be easily adapted from the equivalent result in type C.
  We highlight any notable differences and state analogues of our main conjecture.
\end{enumerate}

\subsection*{Acknowledgments}
We thank Steven Sam for introducing us to this problem, for many helpful discussions, and for detailed feedback on various drafts of this paper.
The authors were partially supported by NSF grant DMS 2302149.

\nocite{sagemath,gs:macaulay2}

\section{Preliminaries}
\label{sec:prelim}

Unless noted otherwise, we work over an algebraically closed field of characteristic zero.
Without loss of generality, we take this field to be $\bC$.

\subsection{Kalman variety}
Let $V$ be a vector space with $\dim{V} = n$ and fix a $d$-dimensional subspace $L \subseteq V$.
The space of linear operators on $V$ will be denoted $\gl(V)$.
The \inlinedef{Kalman variety} is the space of all endormorphisms of $V$ that fix an $s$-dimensional subspace inside $L$:
\begin{equation}
  \Kal(s, L, V) =
  \Kal(s, d, n) =
  \left\{ \varphi \in \gl(V) \middle| \exists U \subseteq L, \dim{U} = s, \varphi(U) \subseteq U \right\}.
\end{equation}

The space above is an algebraic variety.
To describe its defining equations, first fix a basis of $V$ with respect to the decomposition $L \oplus (V / L)$.
We write any generic $\varphi \in \gl(V)$ as block matrix $\left[\begin{smallmatrix} A & B \\ C & D \end{smallmatrix}\right]$ with respect to this basis and define the \inlinedef{type A reduced Kalman matrix}
\begin{equation}
  \label{def:obs-mat:type-a}
  \Theta^\bA(\varphi) =
  \begin{bmatrix}
    C & C A & \dots & CA^{d - 1}
  \end{bmatrix}^\intercal.
\end{equation}
The zero locus of the $(d - s + 1) \times (d - s + 1)$ minors defines $\Kal(s, d, n)$ \cite[Theorem~4.5]{os2012:matrices}.
In general these equations are not minimal and a description of the minimal equations appears in \cite{hua2020:equations}.

\subsection{Partitions, Schur functors, and decomposition formulae}
A \inlinedef{partition} of $n$ is a weakly decreasing tuple $\lambda = (\lambda_1, \dots, \lambda_k)$ such that $\sum \lambda_i = n$; we write $\lambda \vdash n$ and $|\lambda| = n$.
The \inlinedef{length $l(\lambda)$} of a partition is the largest $i$ such that $\lambda_i \neq 0$.
If $\lambda$ has repeated entries, we use exponential notation to write $\lambda$; for example, a partition of the form $(a,a,b,b,b,c)$ is denoted $(a^2,b^3,c)$.
The \inlinedef{transpose partition $\lambda^\top$} to $\lambda$ is defined by $\lambda^\top_i = \#\{j \, | \lambda_j \geq i\}$.
If $\mu$ and $\lambda$ are partitions then $\mu \subseteq \lambda$ if and only if $\mu_i \leq \lambda_i$ for all $i$.
If $s$ is maximal such that $s \times s = (s^s) \subseteq \lambda$ then we say $\lambda$ has rank $s$ and contains a \inlinedef{Durfee square} of size $s$ \cite[\S 1.8]{Sta:EC1}.
Note that $s$ is the largest index satisfying $\lambda_s \geq s$.

Let $V$ be an $n$-dimensional complex vector space.
Partitions with at most $n$ rows parameterize the irreducible polynomial representations of $\gl(V)$ \cite[Chapter~6]{FH:RepresentationTheory}.
The representation corresponding to the partition $\lambda$ will be denoted $\bS_\lambda V$ where $\bS_\lambda$ is the Schur functor associated to $\lambda$; see \cite[Chapter~2]{Wey:CohomVect} for a definition.
Note that we have changed the notation such that $\bS_\lambda V$ in our notation corresponds to $\bL_{\lambda^\top}V$ in \cite[Chapter~2]{Wey:CohomVect}.
In particular, $\bS_{(1^d)} V \cong \bigwedge^d V$ and $\bS_{(d)} V \cong \mathrm{Sym}^d V$.
The representation dual to $\bS_\lambda V$ will be denoted $\bS_\lambda V^\ast = \bS_{-\lambda^\mathrm{op}}V$ where $-\lambda^{\mathrm{op}} = (-\lambda_n, \dots, -\lambda_1)$.
The functor $\bS_\lambda$ is compatible with base change and hence we can construct $\bS_\lambda \sV$ for any locally free sheaf $\sV$ on a scheme.

We end this subsection with a list of formulae involving Schur functors that will be necessary for our calculations.
All of the material here appears elsewhere, for example in \cite[Section~2.3]{Wey:CohomVect}, however we restate them here for reference and notational consistency.

\begin{proposition}
  Let $E$ and $F$ be vector spaces.
  There are $\GL(E) \times \GL(F)$ equivariant isomorphisms:
  \begin{gather}
    \label{p:schur-sum}
    \bS_{\lambda}(E \oplus F) = \bigoplus_{\mu \subseteq \lambda} \bigoplus_{|\nu| = |\lambda| + |\mu|} \LR{\nu}{\lambda, \mu} \bS_{\nu} E \otimes \bS_{\mu}F
    \allowdisplaybreaks \\
    \label{p:cauchy}
    \Sym^n (E \otimes F) =  \bigoplus_{|\lambda| = n} \bS_\lambda E \otimes \bS_\lambda F,
    \quad
    \bigwedge^n(E \otimes F) =  \bigoplus_{|\lambda| = n} \bS_\lambda E \otimes \bS_{\lambda^\top} F,
    \allowdisplaybreaks \\
    \label{p:lr-rule}
    \bS_\lambda E \otimes \bS_\mu E = \bigoplus_{|\nu| = |\lambda| + |\mu|} \LR{\nu}{\lambda,\mu} \bS_\nu E,
    \allowdisplaybreaks \\
    \label{p:wedge-sym}
    \bigwedge^m (\Sym^2E) = \bigoplus_{\lambda \in Q_1(2m)} \bS_\lambda E.
  \end{gather}
  The $\LR{\gamma}{\alpha, \beta}$ denote the Littlewood--Richardson coefficients and $Q_1(2m)$ consists of partitions
  \begin{equation*}
    \lambda = (d + 1 + \alpha_1, \dots, d + 1 + \alpha_s, \alpha^\top_1, \dots, \alpha^\top_d)
  \end{equation*}
  of $2m$ where $\alpha = (\alpha_1, \dots, \alpha_d)$ is a partition satisfying $\alpha_1^\top \le d$ for some $d \geq 0$.
\end{proposition}
\begin{proof}
  See
  \cite[Proposition~2.3.1(a), Theorem~2.3.6]{Wey:CohomVect} for \eqref{p:schur-sum},
  \cite[Corollary~2.3.3]{Wey:CohomVect} for \eqref{p:cauchy},
  \cite[Theorem~2.3.4]{Wey:CohomVect} for \eqref{p:lr-rule},
  and \cite[Proposition~2.3.9(a)]{Wey:CohomVect} for \eqref{p:wedge-sym}.
\end{proof}

The decomposition~\eqref{p:cauchy} is called the Cauchy Rule.
Similar formulas also exist for $\bigwedge^m(\bigwedge^2 E)$ and  $\Sym^m(\Sym^2 E)$ see, respectively, \cite[Proposition~2.3.9(b)]{Wey:CohomVect} and \cite[Proposition~2.3.8]{Wey:CohomVect}.

\begin{remark}
  \label{r:q1-parts}
  The partitions appearing in $Q_1(2m)$ can be visualized as
  \begin{equation*}
    \begin{tikzpicture}
  \begin{scope}[cm={1, 0, 0, -1, (0, 0)}, scale=0.6]
    \draw
    (0, 2) node[left] {$\lambda = $}
    (1, 1) node {$d \times d$}
    (3.5, 1) node {$\alpha$}
    (1, 3) node {$\alpha^\top$}
    (2.25, -0.2) node[above] {\footnotesize 1}
    ;

    \begin{scope}[|-|, yshift=-5pt]
      \draw (2, 0) -- (2.5, 0);
    \end{scope}

    \draw
    (0, 0) rectangle (2, 2)
    (2, 0) rectangle (2.5, 2);

    % alpha
    \begin{scope}[scale=0.5, cm={1, 0, 0, 1, (5, 0)}]
      \draw
      (0, 0) -- (0, 4) -- (1, 4) -- (1, 3) -- (4, 3) -- (4, 0) -- (0, 0); 
    \end{scope}

    \begin{scope}[scale=0.5, cm={0, 1, 1, 0, (0, 4)}]
      \draw
      (0, 0) -- (0, 4) -- (1, 4) -- (1, 3) -- (4, 3) -- (4, 0) -- (0, 0); 
    \end{scope}
  \end{scope}
\end{tikzpicture}
%%% Local Variables:
%%% mode: LaTeX
%%% TeX-master: "../2025-gkp-isotropic-kalman"
%%% End:

  \end{equation*}
  and we have $|\lambda| = 2m = d^2 + d + 2|\alpha|$. Observe that $d$ is the size of the Durfee square of $\lambda$.
\end{remark}

\subsection{Representations of the symplectic group and branching}
Let $V$ be a vector space of dimension $2n$ with the standard basis $\{e_1, \dots, e_n, e_{-n}, \dots e_{-1}\}$ and consider a skew-symmetric nondegenerate bilinear form on $V$ generated by $\langle e_i, e_j \rangle = \delta_{i,-j}$ for $i > 0$.
A basis of $V$ that satisfies this property is called a \inlinedef{symplectic basis} and the bilinear form is defined by
\begin{equation}
  \label{def:symplectic-form}
  J = \begin{bmatrix} 0 & I^\tau \\ -I^\tau & 0 \end{bmatrix},
  \qquad
  I^\tau = \begin{bmatrix}  & & 1 \\ & \iddots & \\ 1 & & \end{bmatrix}.
\end{equation}
For $v, w \in V$, we have $\langle v, w \rangle = v^\intercal J w$.

Let $U \subseteq V$ be a subspace and define its \inlinedef{symplectic complement} as the subspace
\begin{equation*}
  U^\perp = \left\{ v \in V \mid \langle v, u \rangle = 0 \text{ for all } u \in U \right\}.
\end{equation*}
The subspace $U$ is called \inlinedef{isotropic} if $U \subseteq U^\perp$, \inlinedef{coisotropic} $U^\perp \subseteq U$, and \inlinedef{Lagrangian} if $U = U^\perp$.
If $U$ is Lagrangian then $\dim{U} = n$ is the maximum possible dimension of any isotropic subspace.

The symplectic group $\Sp(V)$ is the group of automorphisms of $V$ that preserve the form $J$ that is,
\begin{align*}
  \Sp(V)
  &= \{ g \in \GL(V) \mid g^\intercal J g = J\} \\
  &= \{ g \in \GL(V) \mid \langle gv ,  gw \rangle = \langle v, w \rangle \text{ for all } v,w \in V \}.
\end{align*}
The Lie algebra $\sp(V)$ consists of block matrices
\begin{equation}
  \label{eq:symplectic-blocks}
  \varphi = \begin{bmatrix} A & B \\ C & D \end{bmatrix}, \quad
  A = -D^\tau, B = B^\tau, C = C^\tau,
\end{equation}
where $A, B, C, D$ are all $n \times n$ matrices and $M^\tau$ denotes the transpose of $M$ with respect to the antidiagonal, see \cite[\S1.2, p.~3]{Hum:IntroductionLie}.

Analogous to the case of the general linear group, there is a bijective correspondence between the irreducible representations of $\Sp(V)$ and dominant weights \cite[Theorem~17.11 and Corollary~17.21]{FH:RepresentationTheory}.
In the case of the symplectic group, the dominant weights are characterized by tuples $\mu = (\mu_1, \cdots, \mu_n)$ satisfying $\mu_1 \geq \mu_2 \geq \cdots \geq \mu_n \geq 0$.
We will denote the representation with highest weight $\mu$ by $\bS_{[\mu]}(V)$

A symplectic vector bundle over a scheme $X$ is a vector bundle $\sV$ equipped with a symplectic form $\langle\cdot , \cdot \rangle \colon \bigwedge^2 \sV \to \O_X$.
The functors $\bS_{[\mu]}$ are compatible with base change and extend to symplectic vector bundles.

Since $\Sp(V) \subseteq \GL(V)$, an irreducible representation $\bS_\lambda V$ of $\GL(V)$ restricts to a representation $\res^{\GL(V)}_{\Sp(V)} \bS_\lambda V$ of $\Sp(V)$.
The decomposition of $\res^{\GL(V)}_{\Sp(V)} \bS_\lambda$ in terms of irreducible representations of the symplectic group is given by the branching rule.

\begin{theorem}[{\cite[Theorem~2.5.1]{kt1987:youngdiagrammatic}}]
  \label{t:branching}
  If $\bS_\lambda V$ is an irreducible representation of $\GL(V)$ then its restriction to $\Sp(V)$ decomposes as
  \begin{equation*}
    \res^{\GL(V)}_{\Sp(V)} \bS_\lambda V
    = \bigoplus_{\mu} \Pi_{\Sp(V)}(\bS_{[\mu]} V)^{\oplus m(\lambda, \mu)}
  \end{equation*}
  where $m(\lambda,\mu) = \sum_{\kappa} c^{\lambda}_{\mu,(2\kappa)^\top}$, $2\kappa$ is the partition $(2 \kappa_1, 2 \kappa_2, \dots)$, and $\Pi_{\Sp(V)}$ is the specialization map on irreducible $\Sp(V)$-representations.
\end{theorem}

If $\Pi_{\Sp(V)}(\bS_{[\lambda]} V) = \bS_{[\mu]} V$  then $\mu$ is obtained from $\lambda$ in the following manner.
Let $s = l(\lambda^\top)$ and for each $i = 1, \dots, s$ ``fold up'' the $i$-the column of $\lambda$ at depth $n + i$ and remove overlapping boxes.
If $k_i$ is the (possibly negative) length of the $i$-th column after removing the overlapping boxes and $t_i = k_i - (i - 1)$, sorting the $t_i$ in descending order $(t_{i_1}, \dots, t_{i_s})$ lets us define $\mu$ by $\mu^\top_j = t_{i_j} + (j - 1)$.
See \cite{kt1987:youngdiagrammatic} for details.

\begin{example}[\mbox{\cite[Example~2]{kt1987:youngdiagrammatic}}]
  Let $n = 3$ and $\lambda^\top = (8,6,4,1)$.
  The folding operation is given by
  \begin{equation*}
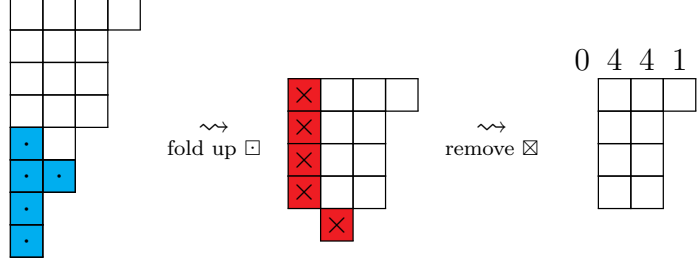
    \begin{array}{ccccc}
  \ydiagram{4, 3, 3, 3, 2, 2, 1, 1}
  *[*(Cyan) \cdot]{0, 0, 0, 0, 0 + 1, 2, 1, 1}
  &
    \underset{\text{fold up } \boxdot}{\leadsto}
  &
    \begin{ytableau}
      \none & \none & \none & \none  \\
      \none & \none & \none & \none  \\
      *(Red)\times & ~ & ~ & ~ \\
      *(Red)\times & ~ & ~ \\
      *(Red)\times & ~ & ~ \\
      *(Red)\times & ~ & ~ \\
      \none & *(Red)\times
    \end{ytableau}
  &
    \underset{\text{remove } \boxtimes}{\leadsto}
  &
    \begin{ytableau}
      \none[0] & \none[4] & \none[4] & \none[1]  \\
      \none & ~ & ~ & ~ \\
      \none & ~ & ~ \\
      \none & ~ & ~ \\
      \none & ~ & ~
    \end{ytableau}
\end{array}
% \ydiagram{4, 3, 3, 3, 2, 2, 1, 1}
% *[*(Yellow)]{0, 0, 0, 0, 0 + 1, 2, 1, 1}
% *[*(Red)\times]{1, 1, 1, 1, 1 + 1}
%%% Local Variables:
%%% mode: LaTeX
%%% TeX-master: "../2025-gkp-isotropic-kalman"
%%% End:

  \end{equation*}
  where blue boxes ($\boxdot$) represent the boxes that get folded  and the overlapping boxes after folding operation are shown in red ($\boxtimes$).
  Removing these boxes gives us $(k_1, k_2, k_3, k_4)= (0,4,4,1)$, $(t_1, t_2, t_3, t_4) = (0,3,2,-2)$, $\mu^\top = (3,3,2,1)$, and $\Pi_{\Sp(V)}(\bS_{[\lambda]}) = (-1)^{4} \bS_{[\mu]} = \bS_{[\mu]}$.
\end{example}

If $\lambda$ is a partition, we can show that the above operation actually results in a subpartition $\mu$ of $\lambda$.
The proof follows form the combinatorics of fold and remove operations above and results in \cite{kt1987:youngdiagrammatic}; we omit details.
\begin{lemma}
  If $\Pi_{\Sp(V)}(\bS_{[\lambda]}) = \bS_{[\mu]}$ then $\mu \subseteq \lambda$.
\end{lemma}

\begin{remark}[Stable branching]
  When $l(\lambda) \leq n$ $\Pi_{\Sp(V)}(\bS_{[\mu]}) = \bS_{[\mu]} V$ and we say the branching is stable and
  \begin{equation*}
    \res^{\GL(V)}_{\Sp(V)} \bS_\lambda V = \sum_\mu m(\lambda,\mu) \bS_{[\mu]} V.
  \end{equation*}
\end{remark}

\subsection{Borel--Weil--Bott}
We state the version of these results as they appear in \cite[Chapter~4]{Wey:CohomVect}; original proofs come from \cite{bot1957:homogeneous, ser1954:representations}.

\subsubsection{Symplectic Borel--Weil--Bott}
Let $V$ be a symplectic vector space.
The isotropic Grassmannian is a projective variety $X = \IGr(s, V) \subseteq \Gr(s, V)$ consisting of all $s$-dimensional isotropic subspaces of a $V$.
The Weyl group of $\Sp(V)$ is the group $W_n = \mathbf Z_2 \wr \fS_n$ of signed permutations.
If $\sR$ is the tautological subbundle on $X$, the quotient $\sR^\perp / \sR$ is a symplectic bundle.
If $\lambda = (\lambda_1, \dots, \lambda_s)$ and $\mu = (\mu_1, \dots, \mu_{n - s})$ then we define the bundle
\begin{equation*}
  \sV(\lambda, \mu) = \bS_\lambda \R \otimes \bS_{[\mu]}(\R^\perp / \R)
\end{equation*}
whose cohomology on $X$ is given by Borel--Weil--Bott.

\begin{theorem}[{\cite[Corollary~4.3.7]{Wey:CohomVect}}]
  \label{t:bwb:iso}
  Suppose $\lambda$ is a dominant weight of $\GL(s)$ and $\mu$ a dominant weight of $\Sp(2n - 2s)$.
  Let $\gamma = (-\lambda_s, \dots, -\lambda_1, \mu_1, \dots, \mu_{n - s})$.
  One of the following mutually exclusive possibilities occurs:
  \begin{enumerate}
  \item
    There exists $\sigma \in W_n$, $\sigma \neq \mathrm{id}$ such that $\sigma \bullet \gamma = \gamma$, in which case, the cohomology groups $\rH^i(\IGr(s, V), \sV(\lambda, \mu)) = 0$ for all $i \geq 0$.

  \item
    There exists a unique $\sigma \in W_n$ such that $\sigma \bullet \gamma = \alpha$ is a dominant integral weight for $\Sp(V)$.
    Then $\rH^i(\IGr(s, V), \sV(\lambda, \mu)) = 0$ for all $i \neq \ell(\sigma)$ and $\rH^{\ell(\sigma)}(\IGr(s, V), \sV(\lambda, \mu)) = \bS_{[\alpha]} V$.
  \end{enumerate}
\end{theorem}

Above, the dot action $\sigma \bullet \mu$ of $\sigma \in W_n$ on a weight  $\gamma = (\gamma_1, \dots, \gamma_n)$ is given by $\sigma \bullet \gamma = \sigma (\gamma + \rho) - \rho$ where $\rho = (n, n - 1, \dots, 1)$.
The Bruhat length $\ell(\sigma)$ of $\sigma$ is calculated using \eqref{def:len-type-c}.

\subsubsection{Relative Borel--Weil--Bott} Now let $X$ be a nonsingular projective variety and $\sE$ a vector bundle of rank $n$ over $X$ with structural morphism $h \colon \sE \to X$.
Let $\Flag_X(b_1, \dots, b_t; \sE)$ be a the relative flag variety over $X$; the structural morphism is again denoted by
\allowbreak$h \colon \Flag_X(b_1, \dots, b_t;\;\sE) \to X$.
Let $\sR_{b_j}$ denote the tautological bundle of rank $b_j$ over $\Flag_X(b_1, \dots, b_t; \sE)$ with $\sR_{b_{t + 1}} = \sE$ and $\sR_{b_0} = 0$.
For weights $\alpha^0, \dots, \alpha^t$ define the vector bundle
\begin{equation*}
  \sV(\alpha) = \sV(\alpha^0, \dots, \alpha^t) = \bigotimes_{j = 0}^t \bS_{\alpha^j} (\sR_{b_{j + 1}} / \sR_{b_j})
\end{equation*}
on $\Flag_X(b_1, \dots, b_t;\; \sE)$.

\begin{theorem}[{\cite[Theorem~4.1.8]{Wey:CohomVect}}]
  \label{t:bwb:rel}
  Let $\alpha^0, \dots, \alpha^t$ be dominant weights and let $\alpha = (\alpha^0, \dots, \alpha^t) = (\alpha^0_1, \dots, \alpha^0_{b_1}, \dots, \alpha^t_1, \dots \alpha^t_{n - b_t})$.
  One of the two mutually exclusive possibilities occurs
  \begin{enumerate}
  \item
    There exists $\sigma \in \fS_n$ such that $\sigma \bullet \alpha = \alpha$.
    Then the higher direct images $\rR^i h_\ast \sV(\alpha)$ all vanish for $i \geq 0$.

  \item
    There exists unique $\sigma \in \fS_n$ such that $\sigma \bullet \alpha = \beta$ is a partition.
    In this case all higher direct images vanish for $i \neq \ell(\sigma)$ and $\rR^{\ell(\sigma)} h_\ast \sV(\alpha) = \bS_\beta \sE$.
  \end{enumerate}
\end{theorem}

Above we have $\sigma \bullet \alpha = \sigma(\alpha + \rho) - \rho$  for $\rho = (n - 1, \dots, 0)$ and the Bruhat length $\ell(\sigma)$ of $\sigma$ is the number of inversions of $\sigma \in \fS_n$.

\subsubsection{Combinatorics of signed permutations}
If $(x_1, \dots, x_n) \in \bZ^n$ then define the sets
\begin{align*}
  \Inv(x_1, \dots, x_n)
  &= \left\{ (i, j) \mid 1 \leq i < j \leq n \text{ and } x_i > x_j \right\},
  \\
  \Neg(x_1, \dots, x_n)
  &= \left\{i \in [n] \mid x_i < 0 \right\}, \\
  \Nsp(x_1, \dots, x_n)
  &= \left\{ (i, j) \mid 1 \leq i < j \leq n \text{ and }  x_i + x_j < 0 \right\}
\end{align*}
called inversions, negative entries, and negative pairs of the tuple $(x_1, \dots, x_n)$.
If $\sigma \in W_n$ then we define the statistics above for $\sigma$ as the corresponding statistics for the tuple $(\sigma(n), \dots, \sigma(1))$; that is, $\Inv(\sigma) = \Inv(\sigma(n), \dots, \sigma(1))$, etc.
The Bruhat length $\ell(\sigma)$ of the signed permutation $\sigma$ is
\begin{equation}
  \label{def:len-type-c}
  \ell(\sigma) = |\Inv(\sigma)| + |\Neg(\sigma)| + |\Nsp(\sigma)|.
\end{equation}
See discussion in \cite[\S2.3]{sam2015:homology} for details on the formulas above and \cite[Appendix~A3]{BB:CombinatoricsCoxeter} for the general theory.

\subsection{Geometric method for computing syzygies}
We end with an overview of Weyman's geometric method for calculating syzygies.
Details and proofs are available in \cite[Chapter~5]{Wey:CohomVect}.
Let $X$ be a projective variety and $E$ a vector space.
Picking a subbundle $Z \subseteq X \times E $ defines a subvariety $Y \subseteq E$ by restricting the projection $\pi = \pi_E \colon X \times E \to E$ to $Z$:
\begin{equation*}
  \begin{tikzcd}
    Z \ar["\subseteq"]{r} \ar["\pi'"]{d} & X \times E  \ar["\pi = \pi_E"]{d} \ar["\pi_X"]{r} & X \\
    Y \ar{r} & E
  \end{tikzcd}
\end{equation*}
We consider $X \times E$ as the total space of the trivial bundle $\sE$ and $Z$ as the total space of the subbundle $\sS \subseteq \sE$.
Let $\xi = (\sE / \sS)^\ast$ and $\eta = \sS^\ast$ be the dual bundles and $A = \Sym(E^\ast)$ the coordinate ring of $E$ with grading given by $\deg{E^\ast} = 1$.
For every $i \in \bZ$ we define the graded $A$-modules
\begin{equation*}
  \bF_i =  \bigoplus_{j \geq 0}\, \rH^j(X, \bigwedge^{i + j}\xi) \otimes A(-i-j).
\end{equation*}

\begin{theorem}[Geometric method]
  \label{t:geom-method}~\newline
  \begin{enumerate}
  \item
    There exists differentials $d_i \colon \bF_i \to \bF_{i - 1}$ of degree 0 such that $\bF_\bullet$ is a complex of graded free $A$-modules and $\rH_{-i}(\bF_\bullet) = \rR^i\pi'_\ast\O_Z$.
  \item
    If $\pi'$ is birational and $\rR^i\pi'_\ast\O_Z = 0$ for all $i > 0$ then $\bF_\bullet$ is a minimal $A$-free resolution of the normalization of $Y$.
    Moreover, the normalization of $Y$ has rational singularities.

  \item
    The sheaf $\rR^i\pi'_\ast\O_Z = \rH^i(Z,\O_Z)$ and can be identified with the graded $A$-module $\rH^i(X,\mathrm{Sym} \, \eta)$
  \end{enumerate}
\end{theorem}
\begin{proof}
  The claims follow from \cite[Theorem~5.1.2]{Wey:CohomVect} and
  \cite[Theorem~5.1.3]{Wey:CohomVect}.
\end{proof}

% \begin{remark}
%   A locally free resolution of $\O_Y$ as an $\O_{X \times E}$-module is given by the Koszul complex
%   \begin{equation*}
%     \bK(\xi)_\bullet \colon 0 \to \bigwedge^t (\pi_X^\ast \xi) \to \dots \to \bigwedge^2 (\pi_X^\ast \xi) \to \pi_X^\ast\xi \to \O_{X \times E}
%   \end{equation*}
%   where  the differentials in the complex are homogeneous of degree $1$ in the coordinate functions on $E$.
%   The direct image $p_\ast \O_Y$ can be identified with the sheaf of algebras $\Sym(\sS^\ast)$.
%   It is often necessary to consider $\bK(\xi)_\bullet$ over $X$.
%   As $E$ is affine, the projection $\pi_X \colon X \times E \to X$ is an affine morphism and the direct image functor $(\pi_X)_\ast$ is exact \cite[Tag 0AVW]{stacks}.
%   Using the projection formula \cite[Exericse~II.5.1]{Har:AG},
%   \begin{equation*}
%     (\pi_X)_\ast \left( \bigwedge^q \pi_X^\ast \xi \right) = \bigwedge^q \xi \otimes \Sym(E^\ast) \otimes \O_X.
%   \end{equation*}
%   This gives us the necessary complex over $X$.
%   Since $(\pi_X)_\ast \bK(\xi)_\bullet$ is a Koszul complex over $X$, the differentials are given by comultiplication.
% \end{remark}

\section{Varieties defined by restriction of bundles}
\label{sec:restriction-defn}

\subsection{Kalman varieties via vector bundles}
We provide an alternate definition of the Kalman variety in terms of restrictions of cotangent bundles on Grassmannians.
Fix a vector space $V$ and a subspace $L \subseteq V$.
Let $X_0 = \Gr(s, L)$ and $Y_0 = \Gr(s, V)$.
The containment $L \subseteq V$ gives us an inclusion $\iota_0 \colon X_0 \hookrightarrow Y_0$ of the corresponding Grassmannians.
So if $0 \to \R_{Y_0} \to  V \times Y_0 \to \Q_{Y_0} \to 0$ is the universal sequence on $Y_0$, the cotangent bundle $\Omega_{Y_0} = \R_{Y_0} \otimes \Q_{Y_0}^\ast$ restricts to
\begin{equation*}
  \xi_0
  = \iota_0^\ast\Omega_{Y_0}
  = \R_{X_0} \otimes (\Q_{X_0}^\ast \oplus (V / L)^\ast)
\end{equation*}
on $X_0$.
Recall that $\sR_{X_0}$ and $\sQ_{X_0}$ are, respectively, the tautological sub- and quotient bundles over $\Gr(s,L)$ defined by \eqref{def:R} and \eqref{def:Q}.
Since $\xi_0$ is a subbundle of the trivial bundle $\O_{X_0} \otimes \End(V)$, let $\eta_0 = (\O_{X_0} \otimes \End(V)) \big/ \xi_0$.
Let $Z$ be the total space of the subbundle $\eta^*$ of $\Gr(s, L) \times \End(V)$.
The Kalman variety $\Kal(s, d, n) \subseteq \End(V)$ is precisely the image of the projection $\pi \colon \Gr(s, L) \times \End(V) \to \End(V)$ restricted to $Z$, see \cite[p.~3]{hua2020:equations}.

\subsection{Isotropic Kalman varieties}
We now follow the same strategy to define a symplectic generalization of Kalman varieties.

Let $V$ be a vector space endowed with a symplectic form.
Suppose $\dim{V} = 2(n + r)$ and let $\{e_1, \dots, e_{n+r}, e_{-n-r}, \dots e_{-1} \}$ be a symplectic basis for $V$ with respect to which the matrices in the symplectic lie algebra $\sp(V)$ are given by
\begin{equation*}
  \varphi = \begin{bmatrix} A & B \\ C & D \end{bmatrix}, \quad
  A = -D^\tau, B = B^\tau, C = C^\tau.
\end{equation*}
Above $A, B, C, D$ are all $(n + r) \times (n + r)$ matrices and $M^\tau$ denotes the transpose of $M$ with respect to the antidiagonal; see \ref{eq:symplectic-blocks}.

Let $W\subseteq V$ be an $r$-dimensional isotropic subspace and $L$ a symplectic subspace such that $W^\perp=L\oplus W$.
The subspace $L$ is a $2n$-dimensional symplectic subspace.
Consider the isotropic Grassmannians $X = \IGr(m, L) = \IGr(m, 2n)$ and $Y = \IGr(m + r, V) = \IGr(m + r, 2n + 2r)$.
Let $\iota = \iota_W \colon X \to Y$ be the inclusion defined by $R \mapsto R \oplus W$.
The cotangent bundle $\Omega_Y$ of $Y$ appears in the short exact sequence
\begin{equation}
  \label{eq:iso-cotangent-ses}
  0 \to \Sym^2\R_Y \to \Omega_Y \to \R_Y \otimes (\R_Y^\perp / \R_Y)^\ast \to 0
\end{equation}
\cite[Chapter~4, Exercise~9]{Wey:CohomVect}.
As pullback of vector bundles is exact and $\iota^\ast \R_Y = \R \oplus \sW$, we obtain a short exact sequence
\begin{equation}
  \label{eq:iso-cotangent-ses-restr}
  0 \to \Sym^2(\R \oplus \sW) \to \iota^\ast \Omega_Y \to (\R \oplus \sW) \otimes (\R^\perp / \R)^\ast \to 0
\end{equation}
of bundles on $X$.
Above, $\R = \R_X$ is the tautological subbundle on $X$ and $\sW$ is the trivial bundle with fibers isomorphic to $W$.
Let $\xi = \iota^\ast \Omega_Y$ and $\sE = \O_{\IGr(m, L)} \otimes \sp(V)$ be a trivial bundle over $\IGr(m, L)$ whose total space is $\IGr(m, L) \times \sp(V)$.
The dual $\sT = \xi^\ast$ is a quotient of $\sE$ by a subbundle $\sS$ resulting in a short exact sequence
\begin{equation*}
  0 \to \sS \to \O_{\IGr(m, L)}\otimes \sp(V) \to \sT \to 0.
\end{equation*}
The total space of $\sS$ is $Z = \Spec(\Sym(\eta))$ where $\eta = \sS^\ast$.

\begin{defn}[Isotropic Kalman variety]
  \label{def:ikal}
  The isotropic Kalman variety $\IKal(m)$ is the image of the projection $\pi$ restricted to $Z$ in the diagram
  \begin{equation}
    \label{eq:ikal-basic-diag}
    \begin{tikzcd}
      Z  \ar["\subseteq"]{r} \ar["\pi' = \left.\pi\right|_Z"]{d} & \IGr(m, L) \times \sp(V)  \ar["\pi = \pi_2"]{d} \ar["\pi_1"]{r} & \IGr(m, L) \\
      \IKal(m) \ar["\subseteq"]{r} & \sp(V)
    \end{tikzcd}.
  \end{equation}
  We call the bundle $\xi$ the defining bundle of $\IKal(m)$.
\end{defn}

We will denote the structure sheaf of $\IKal(m)$ by $\O_m$ and that of its normalization by $\widetilde \O_m$.

\begin{remark}[Lagrangian Kalman variety]
  \label{r:Lagrangian-kalman}
  Since any $U \in \IGr(m, L)$ is an isotropic subspace of $L$, $\dim{U} \leq n$. When $m = n$, $U$ is a Lagrangian subspace and we call the corresponding isotropic Kalman variety a Lagrangian Kalman Variety.
  For any Lagrangian $U$, $U^\perp = U$. So $\R^\perp / \R$ is trivial and the short exact sequence \eqref{eq:iso-cotangent-ses-restr} gives an isomorphism $\Sym^2(\R \oplus \sW) = \iota^\ast \Omega_Y$.
  We will study this case in detail in \S\ref{sec:Lagrangian}.
\end{remark}

\begin{remark}[Coisotropic Kalman variety]
  \label{r:coiso}
  Replacing all isotropic Grassmannians with coisotropic Grassmannians and carrying out the construction above results in the coisotropic Kalman variety $\CoIKal(m')$.
  Since the isotropic Grassmannian $\IGr(m, L)$ is isomorphic to the coisotropic Grassmannian $\CoIGr(2n - m, L)$, there is an isomorphism   $\IKal(m) \cong \CoIKal(2n - m)$ of the resulting Kalman varieties, see Corollary~\ref{c:coikal-set-theo}.

  Proposition~\ref{p:ikal-set-theo} and Corollary~\ref{c:coikal-set-theo} suggest other natural definitions of type C Kalman varieties obtained by changing the isotropic/coisotropic requirements for the prescribed subspace $P$ and the invariant subspace $U$.
  If we require that $U \subseteq P$, then there are, \emph{a priori}, four possible types of containments.
  If $P$ is isotropic and $U$ coisotropic then $\dim{P} \leq \dim{U}$ so this case is not possible.
  The case that we study corresponds to $P$ coisotropic, $P = W^\perp $ for some fixed $W \in \IGr(r, V)$, and $U$ coisotropic.
  The remaining cases, where $U$ is isotropic, describe well-defined algebraic subsets of $\sp(V)$; however, our methods don't apply to those algebraic sets and we won't be studying them here.
\end{remark}

\begin{remark}
  \label{r:bad-restriction}
  Our definition of isotropic Kalman varieties uses the inclusion $\IGr(m, L) \to \IGr(m + r, V)$.
  While, \emph{a priori}, other inclusions are possible, not all of them are well adapted for performing calculations using the geometric method.
  Similar to the calculation done in the type A case in \cite{hua2020:equations,sam2012:equations}, our calculation relies on the normalization having rational singularities and arbitrary inclusions of $\IGr(m, L)$ may fail to have this property.
  For example take $\iota \colon \IGr(n, L) \to \Gr(n, L)$.
  The cotangent bundle $\Omega_{\Gr(n, L)} = \sR_0 \otimes \sQ_0^\ast$ on $\Gr(n, L)$ restricts to
  \begin{equation*}
    \xi = \iota^\ast \Omega_{\Gr(n, L)} = \Sym^2\sR \oplus \bigwedge^2\sR
  \end{equation*}
  on $\IGr(n, L)$ where $\sR$ is the tautological subbundle on $\IGr(n, L)$.

  When $n = 10$ and $q = 14$ the summand $\bS_{(12, 6, 4, 2, 2, 1, 1)} \sR$ appears with multiplicity $3$ in $\bigwedge^q \xi$.
  For the tuples
  \begin{align*}
    \gamma &=  (0, 0, 0, -1, -1, -2, -2, -4, -6, -12), \\
    \gamma + \rho &=  (10, 9, 8, 6, 3, 2, -1, -4, -11)
  \end{align*}
  we obtain $|\Inv(\gamma + \rho)| = 0$, $|\Neg(\gamma + \rho)| = 2$, and $|\Nsp(\gamma + \rho)| = 13$.
  From \eqref{def:len-type-c} and Theorem~\ref{t:bwb:iso}, we conclude that
  $\rH^{15} \left( \IGr(n, L), \bS_{(12, 6, 4, 2, 2, 1, 1)} \sR \right) = \bS_{[1, 1, 1, 1]}L =  \bigwedge^4 L / \bigwedge^2 L$
  appears as a summand in $\rH^{15}(\IGr(n, L), \bigwedge^{14} \xi)$.
  In particular, the bundle $\bigwedge^q \xi$ has cohomology in degree strictly larger than $q$.
  Thus the normalization of the variety defined by $\xi$ fails to have rational singularities.
\end{remark}

\section{Geometric properties of isotropic Kalman varieties}
\label{sec:geom-props}

In this section, we prove various geometric properties of isotropic Kalman varieties.

\begin{proposition}
  \label{p:ikal-irrd}

  The isotropic Kalman variety $\IKal(m)$ is irreducible.

\end{proposition}
\begin{proof}
  The total space of a vector bundle over an irreducible variety is irreducible, so $\IGr(m, L) \times \sp(V)$ and its subbundle $Z$ are irreducible.
  Since $\IKal(m) = \pi(Z)$ is defined to be the image of an irreducible variety, it is irreducible.
\end{proof}

\subsection{Set theoretic description}

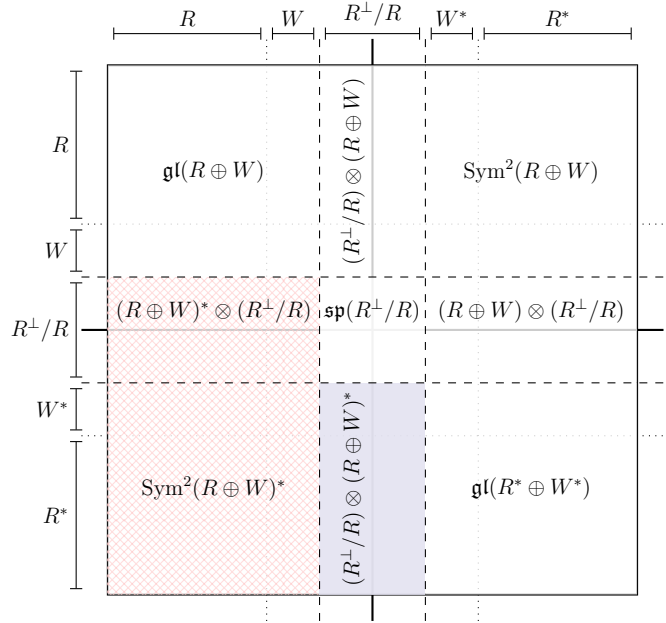
\begin{figure}[h]
  \begin{equation*}
    \begin{tikzpicture}[scale=0.7, every node/.style={scale=0.7}]
  % \draw[step=1cm,ProcessBlue!50,very thin] (-2,-2) grid (12, 12);
  \draw[thick] (0, 0) rectangle (10, 10);

  \draw[thick]
  (5, -0.5) -- (5, 10.5)
  (-0.5, 5) -- (10.5, 5);

  \draw[dotted]
  (3, -0.5) -- (3, 10.5) (-0.5, 3) -- (10.5, 3)
  (7, -0.5) -- (7, 10.5) (-0.5, 7) -- (10.5, 7)
  ;

  \begin{scope}[|-|, yshift=0.6cm, shorten >= 2pt, shorten <= 2pt]
    \draw (0, 10) to node[midway, above] {$R$}       (3, 10);
    \draw (3, 10) to node[midway, above] {$W$}       (4, 10);
    \draw (4, 10) to node[midway, above] {$R^\perp / R$} (6, 10);
    \draw (6, 10) to node[midway, above] {$W^\ast$}     (7, 10);
    \draw (7, 10) to node[midway, above] {$R^\ast$}     (10, 10);
  \end{scope}

  \begin{scope}[|-|, xshift=-0.6cm, shorten >= 2pt, shorten <= 2pt]
    \draw (0, 0) to node[midway, left] {$R^\ast$}     (0, 3);
    \draw (0, 3) to node[midway, left] {$W^\ast$}     (0, 4);
    \draw (0, 4) to node[midway, left] {$R^\perp / R$} (0, 6);
    \draw (0, 6) to node[midway, left] {$W$}       (0, 7);
    \draw (0, 7) to node[midway, left] {$R$}       (0, 10);
  \end{scope}

  \begin{scope}[opacity=0.8, White]
    \fill
    % [Blue!25]
    (0, 0) rectangle (4, 4)
    (6, 6) rectangle (10, 10)
    ;

    \fill
    % [Red!25]
    (0, 6) rectangle (4, 10)
    (6, 0) rectangle (10, 4)
    ;

    \fill
    [White, opacity=0.95]
    (4, 4) rectangle (6, 6);

    \fill
    % [Yellow!25]
    (0, 4) rectangle (4, 6)
    (4, 0) rectangle (6, 4)
    (6, 4) rectangle (10, 6)
    (4, 6) rectangle (6, 10)
    ;

    \fill[pattern=crosshatch, pattern color=Red!20]
    (0, 0) rectangle (4, 6);

    \fill[Blue!10, opacity=0.8]
    (4, 0) rectangle (6, 4);
  \end{scope}

  \draw[dashed]
  (4, -0.5) -- (4, 10.5) (-0.5, 4) -- (10.5, 4)
  (6, -0.5) -- (6, 10.5) (-0.5, 6) -- (10.5, 6)
  ;

  \draw
  (2, 8) node {$\gl(R \oplus W)$}
  (2, 2) node {$\Sym^2(R \oplus W)^\ast$}
  (2, 5) node[above] {$(R \oplus W)^\ast \otimes (R^\perp / R)$}

  (5, 2) node[rotate=90, above] {$(R^\perp / R) \otimes (R \oplus W)^\ast$}
  (5, 5) node[above] {$\sp(R^\perp / R)$}
  (5, 8) node[rotate=90, above] {$(R^\perp / R) \otimes (R \oplus W)$}

  (8, 2) node {$\gl(R^\ast \oplus W^\ast)$}
  (8, 5) node[above] {$(R \oplus W) \otimes (R^\perp / R)$}
  (8, 8) node {$\Sym^2(R \oplus W)$}
  ;

\end{tikzpicture}
%%% Local Variables:
%%% mode: LaTeX
%%% TeX-master: "../2025-gkp-isotropic-kalman"
%%% End:
  \end{equation*}
  \caption{%
    \label{fig:set-theo-blocks}
    Block decomposition of an arbitrary symplectic endomorphism $\varphi$ with respect to the decomposition $V = R \oplus W \oplus (R^\perp / R) \oplus W^\ast \oplus R^\ast$.
    The four quadrants defined by the solid lines correspond to $A$, $B$, $C$, and $D$ from \eqref{eq:symplectic-blocks}.
    The colored blocks are used in proof of Proposition~\ref{p:ikal-set-theo} and the decomposition \eqref{eq:lag-sp-blocks} is simply a reorganization of the blocks here.
  }
\end{figure}

\begin{proposition}
  \label{p:ikal-set-theo}
  As a subvariety of $\sp(V), $the isotropic Kalman variety is given by
  \begin{equation*}
    \IKal(m) = \{ \varphi \in \sp(V) \mid \exists\, U \in \IGr(m,L), \varphi(U \oplus W) \subseteq U \oplus W\}.
  \end{equation*}
  In other words, the isotropic Kalman variety consists of all matrices in $\sp(V)$ that fix an $m + r$-dimensional isotropic subspace in $V$ containing $W$.
\end{proposition}

\begin{proof}
  Decompose an arbitrary matrix of $\sp(V)$ using Figure~\ref{fig:set-theo-blocks}.
  Since $\xi = (\sE / \sS)^\ast$, the quotient bundle $(\sE / \sS)$ fits in the short exact sequence \eqref{eq:iso-cotangent-ses-restr} and the total space of $\sS$ consists of matrices where the shaded parts (red, hatched) vanish (note that the solid colored parts in blue are automatically forced to vanish).
  These are precisely the linear maps that send $R \oplus W$ to itself.
  So, the total space $Z$ of the vector bundle $\sS$ is given by
  \begin{equation*}
    Z = \{(U,\varphi) \in \IGr(m,L) \times \sp(V) \mid \varphi(U \oplus W) \subseteq U \oplus W \}.
  \end{equation*}
  Using Definition~\ref{def:ikal} and, in particular, \eqref{eq:ikal-basic-diag},
  \begin{equation*}
    \IKal(m) =
    \pi(Z) =
    \{\varphi \in \sp(V) \, | \, \exists \, U \in \IGr(m,L), \varphi(U \oplus W) \subseteq U \oplus W \}. \qedhere
  \end{equation*}
\end{proof}

\begin{corollary}
  \label{c:coikal-set-theo}
  As a subvariety of $\sp(V)$, the coisotropic Kalman variety is given by
  \begin{equation*}
    \CoIKal(\widetilde m) =
    \{ \varphi \in \sp(V) \;\vert\; \exists \widetilde U \in \CoIGr(\widetilde m,L), \varphi(\widetilde U \oplus W) \subseteq \widetilde U \oplus W\}
  \end{equation*}
  where $n \leq \widetilde m \leq 2n$.
\end{corollary}

\subsection{Defining equations}
Recall that the defining equations for the type A Kalman variety were given by vanishing of minors of the type A reduced Kalman matrix $\Theta^\bA(\varphi)$.
We introduce an analogous criteria in the isotropic case.
Any $\varphi \in \sp(V)$ can be written as a block matrix
\begin{equation}
  \label{eq:lag-sp-blocks}
  \varphi =
  \begin{bmatrix}
    {A_0} & {A_1} & {B_0} \\
    {A_2} & {A_3} & {B_1} \\
    {C_0} & {C_1} & {D_0}
  \end{bmatrix}
\end{equation}
with respect to the decomposition $V=L\oplus W\oplus W^*$.
We define the type C reduced Kalman matrix associated to $\varphi$ as the block matrix
\begin{equation}
  \label{eq:type-c-obs-mat}
  \Theta^{\bC}(\varphi) =
  \begin{bmatrix}
    C_0 &  C_0  A_0 & \cdots &  C_0  A_0^{2n - 1}
  \end{bmatrix}^\intercal.
\end{equation}
Note that $\Theta^\bC(\varphi) \colon L \to \widetilde{W}$ is a $2n \times r (2n -1)$ matrix, where $\widetilde{W} = {(W^*)}^{\oplus 2n  - 1}$.
The rest of this subsection will be spent proving the following main result.

\begin{theorem}\label{t:defn-eq}
  Let $\varphi \in \sp(V)$ and consider its block decomposition as given in \eqref{eq:lag-sp-blocks} and let $\Theta = \Theta^\bC(\varphi)$ be the type C reduced Kalman matrix associated to $\varphi$, then $\varphi \in \IKal(m)$ if and only if the $(m + 1) \times (m + 1)$ minors of $\Theta$  vanish, $C_1 = 0$, and $C_0 A_0^k A_1 = 0$  for $0 \leq k \leq 2n - 1$.
\end{theorem}
Before we give the proof, we first prove some results concerning symplectic vector spaces.

\begin{lemma}
  Suppose $U$ is a symplectic vector space of dimension $2n$ and $\varphi \in \sp(U)$.
  For every $k \leq n$, there exists a subspace $K \in \IGr(k,U)$ such that $\varphi(K) \subseteq K$.
  \label{l:sg-lite}
\end{lemma}
\begin{proof}
  We will prove this by induction on $\dim U$. For the base case ($\dim U = 2$), we must have $k = 1$.
  As $\varphi$ must have an eigenvector in $U$ whose isotropic in $U$, the statement follows.

  For the inductive step, we consider the subspace $R$ spanned by an eigenvector of $\varphi$.
  The symplectic form is well-defined on the symplectic space $R^\perp / R$.
  Further, $R^\perp /R$ is properly contained in $U$.
  By the inductive hypothesis, $R^\perp / R$ must contain an invariant isotropic subspace of dimension $k-1$.
  This subspace must be of the form $K/R$, where $R^\perp \supseteq K \supseteq R$.
  As $K/R$ is isotropic in $R^\perp/R$ and $R$ annihilates $R^\perp$, $K \subseteq U$ is isotropic and $\dim{K} = k - 1 + \dim R = k$.
  Further, because $ K/R$ is invariant under $\varphi$, $K$ is also invariant under $\varphi$.
  This completes the proof.
\end{proof}

\begin{lemma}
  Suppose $U$ is a symplectic vector space of dimension $2n$ and $\varphi \in \sp(U)$.
  If $\varphi(\widetilde{U}) \subseteq \widetilde{U}$ for some $\widetilde{U} \subseteq U$ and there exists $K \in \IGr(k,U)$ satisfying $K \subseteq \widetilde{U}$ for all $k > 0$, then there is a $\varphi$-invariant subspace $K_0 \in \IGr(k,U)$ such that $\varphi(K_0) \subseteq K_0$ and $K_0 \subseteq \widetilde{U}$.
  \label{l:sg}
\end{lemma}
\begin{proof}
  We will first prove the statement by induction on $\dim \widetilde{U}$.
  For the base case, $\dim \widetilde{U} = 1$, $k = 1$ implies $K = \widetilde{U}$.
  Since $\widetilde{U}$ is invariant under $\varphi$, $K$ is also invariant under $\varphi$ and we set $K_0 = K$.

  For the inductive step, we consider the symplectic space $\widetilde{U}/\rad(\widetilde{U})$, where
  \begin{equation*}
    \rad(\widetilde{U})
    = \{w \in \widetilde{U} \mid \langle w, w' \rangle = 0 \text{ for all } w' \in \widetilde{U}\}.
  \end{equation*}
  The vector space $\widetilde{U}/\rad(\widetilde{U})$ is a symplectic space and the restriction of the symplectic form of $U$ on it is well-defined.
  If $\rad(\widetilde{U}) \neq 0$, we can use the inductive hypothesis on $\widetilde{U}/\rad(\widetilde{U})$ to conclude that there must exist an isotropic subspace in $\widetilde{U}/\rad(\widetilde{U})$ that is $\varphi$-invariant and has dimension $k-\dim \rad(\widetilde{U})$.
  This subspace must be of the form $K_0/\rad(\widetilde{U})$, where $K_0$ is invariant under $\varphi$ and belongs to $\IGr(k,U)$.
  This proves the claim in the case where $\rad(\widetilde{U}) \neq 0$.
  If $\rad(\widetilde{U}) = 0$, then $\widetilde{U} = U$ and the claim follows from Lemma~\ref{l:sg-lite}.
\end{proof}

\begin{remark}
  \label{r:co-iso-sup-lag}
  Note that a subspace of a symplectic vector space $U$ being coisotropic is equivalent to it containing a Lagrangian subspace.
  By definition, $K \subseteq U$ is coisotropic if and only if $K^\perp$ is isotropic.
  Moreover, $K^\perp$ is isotropic if and only if it is contained inside a Lagrangian subspace $S$ of $U$.
  Since taking symplectic complements reverses containments and $S = S^\perp$  is Lagrangian, the last condition is equivalent to $S \subseteq K$.
\end{remark}

Based on the remark above, we prove the coisotropic analogue of Lemma \ref{l:sg}.

\begin{lemma}
  Suppose $U$ is a symplectic vector space of dimension $2n$ and $\varphi \in \sp(U)$.
  If $\varphi(\widetilde{U}) \subseteq \widetilde{U}$ for some $\widetilde{U} \in \CoIGr(2n-t,U),$ then, for any $t \leq k \leq n$, there exists a $K \in  \CoIGr(2n-k,U)$ such that $K \subseteq \widetilde{U}$ and $\varphi(K) \subseteq K$.
  \label{l:sg-co-iso}
\end{lemma}
\begin{proof}
  From Remark~\ref{r:co-iso-sup-lag}, we know that $\widetilde{U}$ contains a Lagrangian subspace $S$ of $U$.
  By Lemma~\ref{l:sg}, there exists a Lagrangian subspace $\widetilde{S}$ of $U$ such that  $\widetilde{S} \subseteq \widetilde{U}$ and $\varphi(\widetilde{S}) \subseteq \widetilde{S}$.
  The map $\varphi$ induces an endomorphsim on $\widetilde{U}/\widetilde{S}$ whose associated matrix is similar to an upper triangular matrix.
  Thus, there exists a basis $\{v_1 + \widetilde{S}, \dots, v_{n-t} + \widetilde{S}\}$ of $\widetilde{U}/\widetilde{S}$ such that for each $i$, $\Span\{v_1 + \widetilde{S}, \dots v_i + \widetilde{S}\}$ is an invariant subspace of $\varphi$.

  Now, for all $t \leq k \leq n$, the subspace $\widetilde{S} \oplus \Span \{v_1, \dots v_{n-k}\}$ of $U$ is a $\varphi$-invariant coisotropic subspace of dimension $2n-k$. This completes the proof.
\end{proof}

\begin{remark}
  \label{r:ikal-inclusion}
  After taking symplectic complements, Lemma \ref{l:sg-co-iso} implies that there is a chain of inclusions $\IKal(0) \subseteq \IKal(1) \subseteq \dots \subseteq \IKal(n)$.
  In Theorem~\ref{t:nonnormal} we will show that these inclusions are related to the singular and nonnormal loci of isotropic Kalman varieties.
\end{remark}

\begin{proof}[Proof of {Theorem~\ref{t:defn-eq}}]
  Let $A$ be the block corresponding to the restriction of $\varphi$ to $L \oplus W$ and $C$ be the block corresponding to the induced map from $L\oplus W$ to $W^\ast$.
  Let $\Theta'=\begin{bmatrix} C & CA & \dots & CA^{2n+r - 1} \end{bmatrix}^\intercal$ be the type A reduced Kalman matrix corresponding to $\varphi$.
  Then, the proof of \cite[Theorem~4.5]{os2012:matrices} implies that $\varphi$ fixes $\ker(\Theta')$.
  Matrix multiplication shows that $C_1 = 0$ and $C_0 A_0^k A_1 = 0$ for $0 \leq k \leq 2n - 1$ hold if and only if $W\subseteq \ker(\Theta')$ holds.
  The latter condition is satisfied if and only if $\ker(\Theta') = \ker(\Theta) \oplus W$.

  Furthermore,  the $(m + 1) \times (m + 1)$ minors of $\Theta$ vanish if and only if $\ker(\Theta)$ is at least $(2n-m)$-dimensional.
  Since $\varphi$ fixes $\ker(\Theta')$, it also fixes $\rad(\ker(\Theta'))$ and we have $\rad(\ker(\Theta'))=\rad(\ker(\Theta))\oplus W$.

  Since $\dim \ker(\Theta) \geq 2n - m$,  there is containment $\rad(\ker(\Theta)) \subseteq \ker(\Theta)^{\perp}$; this implies $\dim{\rad(\ker(\Theta))} \leq m$.
  If $\dim \rad(\ker(\Theta)) = m,$ then we are done since the radical of a subspace, by definition, is always isotropic.
  If, on the other hand, the inequality is strict, then $\rad(\ker(\Theta))^\perp\oplus W$ is an invariant coisotropic subspace of $V$ of dimension greater than $2n-m$.
  We are done once we show that this space contains an invariant coisotropic subspace of dimension $2n - m$ that is fixed by $\varphi$; this is precisely the content of Lemma \ref{l:sg-co-iso}.
\end{proof}

\begin{remark}
  \label{r:sym2W=0}
  From Theorem~\ref{t:defn-eq}, we observe that the defining equations include linear equations corresponding to $C_1 = 0$.
  This follows from the fact that $\varphi(W)\subseteq L \oplus W$ and so the variables corresponding to $W \to W^*$ must vanish on isotropic Kalman varieties.
  Thus, these varieties are actually subvarieties of the smaller affine space $\sp(V) / \Sym^2 W$.
  We may define $\IKal(m)$ using the set-up where the trivial bundle is $\sp(V) / \Sym^2 W$ and the quotient bundle $\xi$ is obtained in a similar fashion by taking a quotient by $\Sym^2W$.
  This significantly simplifies the calculations and we will rely on this reduction while working with explicitly examples.
\end{remark}

\subsection{Desingularization and dimension}
Here we show that that morphism \(\pi'\) of \eqref{eq:ikal-basic-diag} is birational and use the ranks of the bundle in \eqref{eq:iso-cotangent-ses-restr} to calculate the dimension of \(\IKal(m)\).
The birationality result will be used in conjunction with the geometric method of Theorem~\ref{t:geom-method} to show that \(\IKal(m)\) has rational singularities.

\begin{proposition}
  \label{p:birational}
  The map $\pi' \colon Z \to \IKal(m)$  is a birational isomorphism.
\end{proposition}
\begin{proof}
  As $\IKal(m)$ is irreducible by Proposition~\ref{p:ikal-irrd}, it suffices to give a nonempty open subset  $\sU \subseteq \IKal(m)$ such that the fiber over each point in $\sU$ is a singleton.
  Let $\sU$ be the complement of $\IKal(m-1) \subseteq \IKal(m)$ (cf. Remark~\ref{r:ikal-inclusion}).
  If $\varphi$ is a matrix in $\sU$ that fixes two distinct subspaces $R \oplus W$ and $R' \oplus W$ where $R$ and $R'$ are $m$-dimensional isotropic subspaces in $L$, then $\varphi$ also fixes their intersection, which is an isotropic subspace of $V$ containing $W$ of dimension at most $m+r-1$.
  This contradicts the fact that $\varphi\in \sU$.
  Thus, any element in $\sU$ has a unique preimage in $Z$.

  We now show that this subset is nonempty.
  Fix some $R \in \IGr(m,L)$ and let $\varphi \in \sp(V)$ be given by
  \begin{equation*}
    \varphi =
    \begin{bmatrix}
      -N^\tau & 0 & 0 \\
      0 & P & 0 \\
      0 & 0 & N
    \end{bmatrix}
  \end{equation*}
  where $N \in \gl(R^*\oplus W^*)$ and $P\in \sp(R^\perp/R)$, see the decomposition given in Figure~\ref{fig:set-theo-blocks}.
  We take $N$ to be an element in the complement of $\Kal(1,R^\ast,(R \oplus W)^*)$ and claim that any such $\varphi$ belongs to $\sU$.
  We know that $\varphi$ fixes $R^\perp \oplus W$ and we want to show that it does not fix any $(2n - m + r + 1)$-dimensional subspace of $L \oplus W$.
  Assume to the contrary that $\varphi$ fixes an $(2n - m + r + 1)$-dimensional subspace $T \subseteq L\oplus W$.
  We may assume that $T$ contains $R^\perp \oplus W$; if not, we consider $T_0 = T+R^\perp+W$, which is also fixed by $\varphi$, and then take $T'$ to be a $(2n-m+r+1)$-dimensional invariant subspace of $T_0$ containing $R^\perp \oplus W$.
  Suppose $T$ is the span of  $R^\perp$, $W$, and $v$.
  Since $T \subseteq L \oplus W = R^\perp\oplus R^*\oplus W$, we may assume that $v\in R^*$.
  From the construction of $\varphi$, this would imply that $v$ is an eigenvector for $N$.
  As $N \notin \Kal(1, R^*, R^*\oplus W^*)$, this is a contradiction.
  This completes the proof.
\end{proof}

\begin{proposition}
  \label{p:codim}
  We have
  \begin{align*}
    \dim{\IKal(m)} &= n + mr + n^2 + (n + r)^2 + {r + 1 \choose 2}, \\
    \codim_{\sp(V)}{\IKal(m)} &= r(2n -m)  + {r + 1 \choose 2}.
  \end{align*}
\end{proposition}
\begin{proof}
  We recall that \(\dim{L} = 2n\), \(\dim{V} = 2(n + r)\), \(\rank(\sR) = m\), and \(\rank(\sW) = r\).
  So from \eqref{eq:iso-cotangent-ses-restr}, we have
  \begin{align*}
    \rank{(\xi)}
    &= \rank{(\Sym^2(\sR \oplus \sW))} + \rank{[(\sR \oplus \sW)\otimes (\sR^\perp / \sR)]} \\
    &= {m + r + 1 \choose 2} + 2(n - m)(m + r).
  \end{align*}
  From Proposition~\ref{p:birational}, we have \(\dim{\IKal(m)} = \dim{Z'}\).
  Since \(Z = \Spec(\Sym(\eta))\),  $\rank{\sp(V)} = \dim{\sp(V)} =(n+r)(2(n+r)+1)$, and \(\dim{\IGr(m, L)} = \dim{\Gr(m, L)} - {m \choose 2} = m(2n - m) - {m \choose 2}\),
  \begin{align*}
    \dim{\IKal(m)}
    &= \rank{\eta} + \dim{\IGr(m, L)} \\
    &= \rank{\sp(V)} - \rank{\xi} + \dim{\IGr(m, L)} \\
    &=  2n^{2} + m r + 2 n r + n+2r^2+r - \frac{r^2}{2} - \frac{r}{2} \\
    &= n + mr + n^2 + (n + r)^2 + {r + 1 \choose 2}.
  \end{align*}
  Since \(\IKal(m) \subseteq \sp(V)\), we have
  \begin{equation*}
    \codim{\IKal(m)}
    = \dim{\sp(V)} - \dim{\IKal(m)}
    = r(2n -m)  + {r + 1 \choose 2}.
    \qedhere
  \end{equation*}
\end{proof}

\begin{remark}
  Following Remark~\ref{r:sym2W=0}, the codimension of $\IKal(m, W,V)$ as a subvariety of \(\sp(V) / \Sym^2(W)\) is \(r (2n  - m)\).
\end{remark}

\subsection{Rational singularities}
Whenever $\pi'$ in \eqref{eq:ikal-basic-diag} is a birational, and $\rR^i \pi'_* \O_m = 0$ for $i > 0$, the geometric method from Theorem~\ref{t:geom-method} gives us a minimal free resolution $\bF_\bullet$ of $\widetilde \O_{m}$ as a $\Sym(\sp(V))$-module.
As $\xi$ was obtained by restricting a cotangent bundle, we use a vanishing result on Dolbeault cohomology from \cite{bro1997:vanishing} together with Theorem~\ref{t:geom-method} to show that $\widetilde\O_{m}$ has rational singularities.

\begin{theorem}
  \label{t:norm-ikal-rational-sing}
  The normalization of the isotropic Kalman variety $\IKal(m)$ has rational singularities.
\end{theorem}
\begin{proof}
  Let $\xi$ be the defining bundle for $\IKal(m)$ so we have $\xi = \iota^*(\Omega_Y)$, where $\Omega_Y$ is the cotangent bundle of $Y = \IGr(m + r, V)$.
  We denote the corresponding quotient $\sp(V) / \xi$ by $\eta$.
  Note that
  \begin{equation}
    \label{eq:eta_K}
    \eta = \sp(V)/\xi = \eta' \oplus (\sR^* \otimes \sW) \oplus (\sW^* \otimes L) \oplus \Sym^2\sW^*,
  \end{equation}
  where $\eta'$ is the quotient bundle in the short exact sequence
  \begin{equation}
    \label{eq:eta-ses}
    0 \to \Omega_X \to \sp(L) \to \eta' \to 0.
  \end{equation}
  In order to show that $\IKal(m)$ has rational singularities, by Theorem \ref{t:geom-method} it suffices to show that $\rH^i(\IGr(m+r,V), \Sym{\eta}) = 0$ for all $i > 0$.
  Using the decomposition of $\eta$ from \eqref{eq:eta_K}, it is enough to check that the cohomology groups $\rH^i(\IGr(m+r,V), \Sym^p \eta' \otimes \Sym^q(\sR^* \otimes \sW))$ vanish for all $p \geq 0$, $q \geq 0$, and $i > 0$.
  Decomposing $\Sym^p\eta \otimes \Sym^q(\sR^* \otimes \sW)$ using Cauchy's formula~\eqref{p:cauchy},
  \begin{equation*}
    \Sym^p \eta' \otimes \Sym^q (\sR^* \otimes \sW)
    = \bigoplus_{|\lambda| = q} \Sym^p \eta' \otimes \bS_\lambda \sR^* \otimes \bS_\lambda \sW.
  \end{equation*}
  So we need to show that $\rH^i(\IGr(m+r,V), \Sym^p \eta' \otimes \bS_\lambda \sR^*) = 0$ for all $i > 0$.

  Using the short exact sequence \eqref{eq:eta-ses}, for any $p > 0$ we obtain
  \begin{multline*}
    0
    \to \dots
    \to \bigwedge^i \Omega_X\, \otimes\, \Sym^{p-i}(\sp(L))
    \to \dots \\
    \dots
    \to \Omega_X\, \otimes \, \Sym^{p - 1}(\sp(L))
    \to \Sym^p(\sp(L))
    \to \Sym^p \eta'
    \to 0.
  \end{multline*}
  Tensoring the above with $\bS_\lambda \sR^\ast$ gives us
  \begin{equation*}
    \dots
    \to \Omega_X\, \otimes\, \Sym^{p-1} (\sp(L)) \otimes \bS_\lambda\sR^*
    \to \Sym^p(\sp(L)) \otimes \bS_\lambda\sR^*
    \to \Sym^p(\eta') \otimes \bS_\lambda \sR^*
    \to 0.
  \end{equation*}
  In order to show that $\rH^i(\IGr(m+r,V), \Sym^p \eta' \otimes \bS_\lambda \sR^*) = 0$ for all $i > 0$, it suffices to show that $\rH^{i + t}(\bigwedge^t \Omega_X \otimes \bS_\lambda \sR^*) = 0$ for all $t \geq 0$ and $i > 0$.
  This last cohomology group is in fact isomorphic to the Dolbeault cohomology
  group $\rH^{i + t, t}(\IGr(m + r, V), \bS_\lambda \sR^*)$ described in \cite[p.~154]{bro1997:vanishing}.
  In particular, the vanishing result \cite[\S2.3, Theorem~2]{bro1997:vanishing} implies that
  \begin{equation*}
    \rH^{i+t}\big(\IGr(m+r,V),\bigwedge^t \Omega_X \otimes \bS_\lambda \sR^*\big) = \rH^{i+t}\big(\IGr(m+r,V), \bigwedge^t \Omega_X \otimes \bS_{(-\lambda_m, \dots, -\lambda_1)} \sR \big) = 0
  \end{equation*}
  whenever $\lambda$ is a dominant weight of $\sp(V)$.
  In our case, $(-\lambda_m, \dots, - \lambda_1)$ is the weight $(\lambda_1, \dots, \lambda_m, 0 \dots 0)$ in the root system of $\sp(V)$.
  Since it is already dominant, the relevant vanishing statement is automatic.
  From Theorem~\ref{t:geom-method}, we conclude that $\IKal(m)$ has rational singularities.
\end{proof}

\subsection{Free resolution and syzygies}
Since $\pi' \colon Z \to \IKal(m)$  is birational (Proposition~\ref{p:birational}) and the higher direct images $\rR^i\pi'_\ast \O_Z $ vanish for all $i > 0$ (Theorem~\ref{t:norm-ikal-rational-sing}), using Theorem~\ref{t:geom-method} we can obtain a description of the free resolution of the normalization of $\IKal(m)$.

\begin{proposition}
  \label{p:norm-ikal-res}
  A minimal free resolution of the normalization of isotropic Kalman variety $\IKal(m)$ defined by $\xi$ is given by
  \begin{equation}
    \bF_i = \bigoplus_{j \geq 0} \rH^j(\IGr(m, L), \bigwedge^{i + j} \xi) \otimes A(-i - j)
    \label{eq:norm-ikal-res}
  \end{equation}
  where $A = \Sym(\sp(V)^\ast)$ is the coordinate ring of $\sp(V)$.
\end{proposition}

We observe, however, that the short exact sequence \eqref{eq:iso-cotangent-ses-restr} is not split in general, but there is a natural filtration on $\wedge^q\xi$ such that the cohomology groups computed with respect to $\gr{\wedge^q\xi}$ determine a (possibly nonminimal) free resolution of the normalization of $\IKal(m)$.

\begin{lemma}
  \label{l:xi-summands}
  The term
  $\bS_\epsilon\R \otimes \bS_{[\delta]}(\R^\perp / \R) \otimes \bS_\alpha\sW$
  appears
  % \begin{equation*}
  %   K(\epsilon, \delta, \alpha) =
  %   \sum_{\substack{(s + t + \bar t) = (q - p)\\ |\mu| = s, |\nu| = p}}
  %   \sum_{\substack{\lambda \in Q_1(2t) \\\bar \lambda \in Q_1(2\bar t)}}
  %   m(\nu^\top, \epsilon)
  %   \LR{\theta}{\lambda, \mu}
  %   \LR{\beta}{\bar \lambda, \mu^\top}
  %   \LR{\nu}{\tau_0, \tau_1}
  %   \LR{\alpha}{\beta, \tau_1}
  %   \LR{\epsilon}{\theta, \tau_0}
  % \end{equation*}
  as a summand of
  \begin{equation*}
    \bigwedge^{q - p} \left[ \Sym^2(\R \oplus \sW) \right]
    \otimes
    \bigwedge^p \left[ (\R \oplus \sW) \otimes (\R^\perp / \R)^\ast \right].
  \end{equation*}
  with multiplicity $K(\epsilon, \delta, \alpha)$ determined by a product of Littlewood--Richardson coefficients.
\end{lemma}
\begin{proof}
  From Lemma~\ref{l:lag-xi-summands} the $\GL(R) \times \GL(W)$-module $\bS_\theta \R \otimes \bS_\beta \sW$ appears in $\bigwedge^{q - p} \Sym^2(\R \oplus \sW)$ with multiplicity
  \begin{equation*}
    K_0 =
    \sum_{\substack{(s + t + \bar t) = (q - p)\\ |\mu| = s}}
    \sum_{\substack{\lambda \in Q_1(2t) \\\bar \lambda \in Q_1(2\bar t)}}
    \LR{\theta}{\lambda, \mu} \LR{\beta}{\mu^\top, \bar \lambda} .
  \end{equation*}
  Using the Cauchy formula \eqref{p:cauchy} we have
  \begin{equation*}
    \bigwedge^p \left[ (\R \oplus \sW) \otimes (\R^\perp / \R)^\ast \right]
    = \bigoplus_{|\nu| = p} \bS_\nu (\R \oplus \sW) \otimes \bS_{\nu^\top}(\R^\perp / \R)^\ast.
  \end{equation*}
  Decomposing the first term using \eqref{p:schur-sum} gives us summands of the form $\bS_{\nu \backslash \tau_1} \R \otimes \bS_{\tau_1} \sW$ appearing with multiplicity $1$ (as long as $\tau_1 \subseteq \nu$).
  Further decomposing $\bS_{\nu \backslash \tau_1}$ using the Littlewood--Richardson rule \eqref{p:lr-rule}, $\bS_{\tau_0} \R \otimes \bS_{\tau_1} \sW$ appears as a summand of $\bigwedge^p \left[ (\R \oplus \sW) \otimes (\R^\perp / \R)^\ast \right]$ with multiplicity $K_1 = \LR{\nu}{\tau_0,\tau_1}$.
  Since $\GL(R^\perp / R)$ and $\Sp(R^\perp / R)$ simultaneously act on fibers of $\bS_{\nu^\top} (\R^\perp / \R)^\ast$, using the branching rule \eqref{t:branching} we get
  \begin{equation*}
    \res^{\GL}_{\Sp} \left(\bS_{\nu^\top} (\R^\perp / \R)^\ast \right) =
    \bigoplus_\delta m(\nu^\top, \delta) \bS_{[\delta]} (\R^\perp / \R)^\ast =
    \bigoplus_\delta m(\nu^\top, \delta) \bS_{[\delta]} (\R^\perp / \R).
  \end{equation*}
  The last equality follows because all polynomial representations of the symplectic group are self-dual.
  One final application of the Littlewood--Richardson rule gives us the final multiplicity
  \begin{equation*}
    K(\epsilon, \delta, \alpha) = \sum m(\nu^\top, \delta) K_0 K_1\LR{\epsilon}{\theta, \tau_1} \LR{\alpha}{\beta, \tau_0}
  \end{equation*}
  of the summand $\bS_{\epsilon} \R \otimes \bS_{[\delta]} (\R^\perp / \R) \otimes \bS_{\alpha} \sW$.
\end{proof}

\begin{remark}
  The indices appearing in \(K(\epsilon, \delta, \alpha)\) come from subposets
  
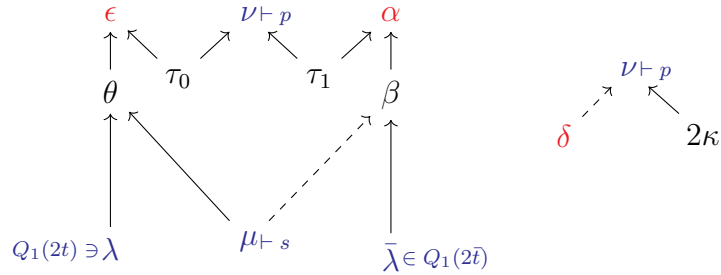
\begin{equation*}
    \begin{tikzpicture}[node distance=0.5]
  \begin{scope}
    % \draw[help lines] (0,0) grid (4,7);
    \coordinate (orig) at (0, -1);
    \draw
    node[above=of orig,Blue]    (R0) {$\lambda$}
    node[above=1.5 of R0]           (R1) {$\theta$}
    node[above=of R1, Red]      (R2) {$\epsilon$}
    node[below right=of R2]     (T0) {$\tau_0$}
    node[above right=of T0,Blue]  (S)  {$\nu$}
    node[below =2.5 of S,Blue](T)  {$\mu$}
    node[below right=of S]     (T1) {$\tau_1$}
    node[above right=of T1,Red] (W2) {$\alpha$}
    node[below=of W2]           (W1) {$\beta$}
    node[below=1.5 of W1,Blue]      (W0) {$\bar\lambda$}
    ;
    \draw[Blue]
    (S)  node[anchor=mid west] {\tiny $\vdash p$}
    (R0) node[anchor=mid east] {\tiny $Q_1(2t) \ni$}
    (W0) node[anchor=mid west] {\tiny $\in Q_1(2\bar t)$}
    (T)  node[anchor=mid west] {\tiny $\vdash s$}
    ;
    \graph{
      (R0) -> (R1) -> (R2);
      (T) -> (R1);
      (T) ->[dashed] (W1);
      (T0) -> {(R2), (S)};
      (T1) -> {(S), (W2)};
      (W0) -> (W1) -> (W2);
    };
  \end{scope}
  \begin{scope}[shift={(6, 1.5)}]
    % \draw[help lines] (0, 0) grid (2, 2);
    \coordinate (orig) at (0, -1);
    \draw
    node[above=of orig,Red] (Q0) {$\delta$}
    node[above right=of Q0,Blue] (S)  {$\nu$}
    node[below right=of S]  (Q1) {$2\kappa$}
    ;
    \graph{
      (Q0) -> [dashed] (S);
      (Q1) -> (S);
    };

    \draw[Blue]
    (S) node[anchor=mid west] {\tiny $\vdash p$}
    ;
  \end{scope}
\end{tikzpicture}
%%% Local Variables:
%%% mode: LaTeX
%%% TeX-master: "../2025-gkp-isotropic-kalman"
%%% End:

  \end{equation*}
  of the Young lattice (see \cite[\S7.2]{Sta:EC2}) where $\beta_0 \dashrightarrow \beta_1$ means $\beta_0^\top \subseteq \beta_1$ and any hat
  
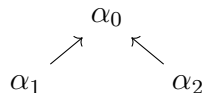
\begin{equation*}
    \begin{tikzpicture}[node distance=0.5]
  \node (a) at (0, 0)          {$\alpha_0$};
  \node[below left=of a]  (a1) {$\alpha_1$};
  \node[below right=of a] (a2) {$\alpha_2$};

  \graph {
    {(a1), (a2)} -> (a);
  };
\end{tikzpicture}
%%% Local Variables:
%%% mode: LaTeX
%%% TeX-master: "../2025-gkp-isotropic-kalman"
%%% End:

  \end{equation*}
  satisfies $|\alpha_0| = |\alpha_1| + |\alpha_2|$.
  The partition $2\kappa$ is coming from the branching rule~\eqref{t:branching}.
\end{remark}

\begin{corollary}
  The bundle $\bigwedge^q\xi$ has a filtration such that the associated graded object $\gr\bigwedge^q\xi$ consists of summands $\bS_\epsilon \sR \otimes \bS_{[\delta]} (\sR^\perp / \sR) \otimes \bS_\alpha \sW$ appearing in the $p$-th graded component with multiplicity $K_{p,q}(\epsilon, \delta, \alpha) = K(\epsilon, \delta, \alpha)$ from Lemma~\ref{l:xi-summands}.
\end{corollary}

\subsection{Nonnormal and singular loci}
\label{ssec:loci}

We end this section with a description of the singular and nonnormal loci of isotropic Kalman varieties.

\begin{theorem}
  \label{t:nonnormal}
  The nonnormal locus for $\IKal(m)$ coincides with the singular locus, and we have
  \begin{equation*}
    \nnormal{(\IKal(m))}
    = \sing{(\IKal(m))}
    = \IKal(m - 1).
  \end{equation*}
\end{theorem}

\begin{proof}
  The main tool we use is Zariski's main theorem \stacks{0AY8} which states that the preimage of a normal point under a birational projective morphism is connected.
  This implies that the closure of the disconnected locus, that is, the points where the preimage is not connected, is contained in the nonnormal locus.

  Recall from Remark~\ref{r:ikal-inclusion} that there is a containment of isotropic Kalman varieties.
  Let $\sU$ be the complement of $\IKal(m-2)$ inside $\IKal(m-1)$.
  The set $\sU$ is open in $\IKal(m-1)$ and can be shown to be nonempty using an argument similar to the one given in the proof of Proposition~\ref{p:birational}.
  Let $ \varphi\in \sU$ and let $T$ be a $(m-1+r)$-dimensional isotropic subspace containing $W$ and fixed by $\varphi$.
  We see that every $(m+r)$-dimensional isotropic subspace containing $W$ that is fixed by $\varphi$ must contain $T$, otherwise its intersection with $T$ would be an invariant isotropic subspace containing $W$ of dimension at most $m-2+r$, which is a contradiction.

  Thus, every invariant $(m+r)$-dimensional isotropic subspace containing $W$ is of the form $T\oplus \Span(v)$ where $v$ is an eigenvector for the endomorphism $\ol\varphi$ induced by $\varphi$ on $T^\perp/T$.
  Generically, the eigenvalues of $\ol\varphi$ on $T^\perp / T$ are distinct.
  Since there are $2(n - m + 1) > 1$ many eigenvectors, the fiber over a generic point of $\IKal(m-1)$ contains finitely many points.
  It is therefore disconnected.

  Observe that we have a chain
  \begin{equation*}
    \IKal(m - 1)
    \subseteq \nnormal{(\IKal(m))}
    \subseteq \sing{(\IKal(m))}
    \subseteq \IKal(m - 1)
  \end{equation*}
  of containments  of closed subsets where the last containment holds because
  \begin{equation*}
    \pi' \colon \pi'^{-1}\left( \IKal(m) \setminus \IKal(m-1) \right) \to \IKal(m) \setminus \IKal(m - 1)
  \end{equation*}
  is an isomorphism and $Z$ is smooth.
\end{proof}

\section{Lagrangian Kalman Varieties}
\label{sec:Lagrangian}

In this section we apply the techniques developed earlier to study Lagrangian Kalman varieties in detail.
That is  we study the case $\IKal(n)$ when $\dim{L} = 2n$.
We end the section by giving the minimal free resolution of the Lagrangian Kalman variety when $\dim L= 2$.

Dimension and codimension calculations follow from Proposition~\ref{p:codim}.
\begin{corollary}
  We have
  \begin{align*}
    \dim \IKal(n) &= n(1 + r) + n^2 + (n + r)^2 + {r + 1 \choose 2},\\
    \codim_{\sp(V)} \IKal(n) &= nr + \binom{r+1}{2}.
  \end{align*}
  \label{c:codim:lagrangian}
\end{corollary}
In the Lagrangian case, some of the conditions on the defining equations of Theorem~\ref{t:defn-eq} end up being extraneous as we show below.
\begin{corollary}
  If $\varphi \in \sp(V)$ has a block decomposition given in \eqref{eq:lag-sp-blocks} then $\varphi \in \IKal(n)$ if and only if $C_1 = 0$ and $ C_0  A_0^k  A_1 = 0$ for all $0 \leq k \leq 2n-1$.
  \label{c:defn-eq:lagrangian}
\end{corollary}

\begin{proof}
  Suppose $\varphi \in \sp(V)$ such that $ C_0  A_0^i  A_1 = 0$ for all $0 \leq i \leq 2n-1$ and $C_1 = 0$.
  Let $\Theta^\bC(\varphi)$ denote the type C reduced Kalman matrix associated to $\varphi$.
  As in proof of Theorem~\ref{t:defn-eq}, $T = \ker \Theta^\bC(\varphi) \oplus W$ is invariant under $\varphi$ and so $\varphi$ descends to an endomorphism on the symplectic vector space $T / \rad(T)$.
  By Lemma~\ref{l:sg-lite}, there exists a $\varphi$-invariant isotropic subspace of every possible dimension in $T/\rad(T)$.
  Say $K$ is one such isotropic subspace and so $K \oplus \rad(T)$ is an isotropic subspace of $V$ fixed by $\varphi$.
  Since $\varphi$ must also fix $(K \oplus \rad(T))^\perp$, by Remark~\ref{r:co-iso-sup-lag}, the latter space contains a Lagrangian subspace of $V$.
  By Lemma \ref{l:sg-co-iso}, $(K \oplus \rad(T))^\perp$ must contain a $\varphi$-invariant Lagrangian subspace of $V$.
  Further, since $W \subseteq \rad(T)$, the Lagrangian is of the form $U\oplus W$ for some $U\in \IGr(n,L)$. This completes the proof for one of the containments.
  The other containment follows from Theorem \ref{t:defn-eq}.
\end{proof}

\subsection{Syzygies of the normalization}
In the Lagrangian case, the bundle $\xi$ is semisimple because $R^\perp/R = 0$ and \eqref{eq:iso-cotangent-ses-restr} is an exact sequence.
In particular, we are able to explicitly calculate (parts of) the minimal free resolution of the normalization of $\IKal(n)$ by directly using Theorem~\ref{t:geom-method} and Borel--Weil--Bott (Theorem~\ref{t:bwb:iso}).

\begin{lemma}
  \label{l:lag-xi-summands}
  If $m = n$, the summands of $\bigwedge^q\xi$ are given by terms of the form
  \begin{equation*}
    c^\theta_{\lambda, \mu} c^\zeta_{\mu^\top, \nu} \bS_\theta\R \otimes \bS_\zeta \sW
  \end{equation*}
  such that $\lambda \in Q_1(2t)$, $|\mu| = s$, $\nu \in Q_1(2u)$ for some
  $s + t + u = q$.
\end{lemma}

\begin{proof}
  First using \eqref{p:wedge-sym} we decompose $\bigwedge^t \Sym^2\R =  \bigoplus_{\lambda \in Q_1(2t)} \bS_\lambda\R$ and $\bigwedge^u \Sym^2 \sW =  \bigoplus_{\nu \in Q_1(2u)} \bS_\nu \sW$.
  The Cauchy formula \eqref{p:cauchy} gives us $\bigwedge^s \R \otimes \sW = \bigoplus_{|\mu| = s} \bS_{\mu}\R \otimes \bS_{\mu^\top}\sW.$
  Combining the terms above using the Littlewood--Richardson rule \eqref{p:lr-rule} we obtain the final decomposition
  \begin{align*}
    \bigwedge^q \xi
    &= \bigwedge^q \left[ (\Sym^2\R) \oplus (\R \otimes \sW) \oplus (\Sym^2 \sW) \right] \\
    &= \bigoplus_{q = s + t + u}  \bigwedge^t\Sym^2\R \otimes \bigwedge^s(\R \otimes \sW) \otimes \bigwedge^u \Sym^2\sW \\
    &=
      \bigoplus_{\substack{q = s + t + u \\ |\mu| = s}}
    \bigoplus_{\substack{\lambda \in Q_1(2t) \\ \nu \in Q_1(2u)}}
    \bigoplus_{\substack{|\theta| = |\lambda| + |\mu| \\ |\zeta| = |\mu| + |\nu|}}
    c^\theta_{\lambda, \mu} c^{\zeta}_{\mu^\top, \nu}\bS_{\theta}\R \otimes \bS_{\zeta}\sW.
    \qedhere
  \end{align*}
\end{proof}

\begin{proposition}
  The term $\bF_0$ of the free resolution of $\widetilde \O_{n}$ is given by
  \begin{equation*}
    \bF_0
    = A \oplus A(-1)^{|Q_{1,n}(2)|} \oplus A(-2)^{\oplus |Q_{1,n}(4)|} \oplus \dots
    = \bigoplus_q A(-q)^{\oplus |Q_{1,n}(2q)|},
  \end{equation*}
  where $Q_{1,n}(2q) = \{\lambda \in Q_1(2q) \mid l(\lambda) \leq n\}$.
\end{proposition}
\begin{proof}
  From \eqref{eq:norm-ikal-res}, the term $\bF_0$ is given by the direct sum of the cohomology groups
  \begin{equation*}
    \bF_0 = \bigoplus_{q \geq 0} \rH^q(\IGr(n, L), \bigwedge^q\xi) \otimes A(-q).
  \end{equation*}
  From Lemma~\ref{l:lag-xi-summands}, the summands of $\bigwedge^q \xi$ are terms of the form $\bS_\theta \sR \otimes \bS_\zeta \sW$ appearing with multiplicity $c^\theta_{\lambda, \mu} c^\zeta_{\mu^\top \nu}$.
  Lemma~\ref{l:f0-terms} shows that these terms have cohomology in degree $q$ if and only if $\theta \in Q_1(2q)$ and $\zeta = \emptyset$.
  Moreover, the cohomology groups are all trivial as representations.

  Finally, if $\theta \in Q_1(2q)$ and $\zeta = \emptyset$ then $\theta = \lambda$, $|\mu| = |\nu| = |\zeta| = 0$ and $c^\lambda_{\lambda, \emptyset} c^\emptyset_{\emptyset, \emptyset} = 1$.
  The claim follows immediately.
\end{proof}

The result above relies on a sequence of combinatorial lemmas involving Borel--Weil--Bott.
\begin{lemma}
  \label{l:s0case}
  Let $\sigma \in W_n$ be a signed permutation, $\lambda \in Q_1(2t)$ and
  $\gamma = (n - \lambda_n, \dots, 1 - \lambda_1) $. If all entries of
  $\sigma(\gamma)$ are positive and strictly decreasing then
  $\ell(\sigma) = t$.
\end{lemma}
\begin{proof}
  Since $\lambda_1 \geq \dots \geq \lambda_n$, we have $n - \lambda_n \geq \dots \geq 1 - \lambda_1$ and $\Inv(\gamma + \rho) = \varnothing$ and $\ell(\sigma) = |\Neg(\gamma + \rho)| + |\Nsp(\gamma + \rho)|$.
  In our setup, the conditions for the sets $\Neg(\gamma + \rho)$ and $\Nsp(\gamma + \rho)$ are equivalent to
  \begin{align*}
    i \in \Neg(\gamma + \rho) &\iff \lambda_i > i, \\
    (i, j) \in \Nsp(\gamma + \rho) &\iff \lambda_i + \lambda_j > i + j.
  \end{align*}
  Since $\lambda \in Q_1(2t)$, there exists a partition $\alpha = (\alpha_1, \dots, \alpha_d)$ such that $\alpha_1^\top \leq d$ and
  \begin{equation*}
    \lambda_i =
    \begin{cases}
      1 + d + \alpha_i  &i \leq d \\
      \alpha^\top_{i - d} &i > d
    \end{cases};
  \end{equation*}
  above, $d$ is the size of the Durfee square of $\lambda$, see Remark~\ref{r:q1-parts}.
  So $\lambda_i > i$ if and only if $i \leq d$ and $|\Neg(\gamma + \rho)| = d$.

  Now pick $(i, j)$ satisfying $1 \leq i < j \leq n$.
  If $i < j \leq d$ then
  \begin{equation*}
    \lambda_i + \lambda_j = \alpha_i + \alpha_j + 2(d + 1) > i + j;
  \end{equation*}
  there are ${d \choose 2}$ such negative sum pairs.
  When $d < i < j \leq n$ we have
  \begin{equation*}
    \lambda_i + \lambda_j = \alpha_{i - d}^\top + \alpha_{j - d}^\top \leq 2\alpha_1 \leq 2d < i + j
  \end{equation*}
  and $(i, j) \notin \Nsp(\gamma + \rho)$.

  Finally suppose $1 \leq i \leq d < j \leq n$.
  Then $\lambda_i + \lambda_j = 1 + d + \alpha_i + \alpha_k^\top$ where $k = j - d$.
  We recall that the Young diagram of $\alpha$ is given by
  \begin{align*}
    D(\alpha) &= \left\{ (i', j') \mid 1 \leq j' \leq \alpha_{i'}, 1 \leq i' \leq l(\alpha) \right\} \\
         &= \left\{ (i', j') \mid 1 \leq i' \leq \alpha_{j'}^\top, 1 \leq j' \leq l(\alpha^\top) \right\}.
  \end{align*}
  So if $(i, k) \in D(\alpha)$ then $\alpha_i + \alpha_k^\top \geq i + k$.
  Conversely if $(i, k) \notin D(\alpha)$ then $\alpha_i < k$, $\alpha^\top_k < i$ and $\alpha_i + \alpha_k^\top < i + k-1$.
  Thus, when $(i, k)\in D(\alpha)$, 
  \begin{equation*}
    \alpha_i + \alpha_k^\top \geq i + k
    \iff
    1 + d + \alpha_i + \alpha_k^\top > i + k + d
    \iff
    \lambda_i + \lambda_j > i + j;
  \end{equation*}
  when $(i, k)\notin D(\alpha)$, 
  \begin{equation*}
    \alpha_i + \alpha_k^\top < i + k-1
    \iff
    1 + d + \alpha_i + \alpha_k^\top < i + k + d
    \iff
    \lambda_i + \lambda_j < i + j.
  \end{equation*}
  We therefore conclude
  \begin{equation*}
    \ell(\sigma)
    = |\Neg(\gamma + \rho)| + |\Nsp(\gamma + \rho)|
    = d + {d \choose 2} + |\alpha| = \frac{d^2 + d + 2|\alpha|}{2}
    = \frac{|\lambda|}{2} = t.
  \end{equation*}
\end{proof}

\begin{lemma}
  \label{l:s>0case}
  Let $\sigma \in W_n$ be a signed permutation, $\theta = (\theta_1, \dots, \theta_n)$ a partition that occurs in the summand in Lemma~\ref{l:lag-xi-summands}, and $\gamma = (- \theta_n, \dots, - \theta_1)$.
  If the entries of $\sigma(\gamma + \rho)$ are all positive and strictly decreasing, then $\ell(\sigma) \leq q$.
  Furthermore, equality holds exactly when $\theta\in Q_1(2q)$, that is, when the weight $\theta$ occurs as a summand of $\bigwedge^q \Sym^2 \R$.
\end{lemma}
\begin{proof}
  Since $\theta_1 \geq \dots \geq \theta_n$, we have $n - \theta_n \geq \dots \geq 1 - \theta_1$, and $|\Inv(\gamma)| = 0$.
  Hence, $\ell(\sigma) = |\Neg(\gamma + \rho)| + |\Nsp(\gamma + \rho)|$.
  So, we want to show that $|\Neg(\gamma + \rho)| + |\Nsp(\gamma + \rho)| \leq q$.

  To simplify our notation, we expand the definition of $\Nsp$ (only for this proof) to also include $\Neg$ by allowing pairs of the form $(i,i)$ if $\lambda_i > i$.
  That is, we define
  \begin{equation*}
    \Nsp(\gamma + \rho) = \{ (i,j) \, | \, i \leq j, \lambda_i + \lambda_j > i + j \}.
  \end{equation*}

  From the proof of Lemma~\ref{l:lag-xi-summands}, $\theta$ is obtained by adding $s$ boxes to a partition $\lambda \in Q_1(2t)$.
  So, we will prove for $s>0$ that $|\Nsp(\gamma + \rho)| < s + t + u$ by carrying out induction on the number of boxes $s$.
  Note that $q = s + t + u$ according to the notation of Lemma~\ref{l:lag-xi-summands}.
  The $s=0$ case follows from Lemma \ref{l:s0case}.

  We first do the base case $s=1$.
  Suppose that $\theta$ is obtained by adding a box to $\lambda\in Q_1(2t)$ in the $k$-th row.
  Let $\gamma'$ denote $(n - \lambda_n, \dots, 1 - \lambda_1)$.
  We will show that the negative sum pairs are unaffected and so $\Nsp(\gamma+\rho) = \Nsp(\gamma'+\rho)$.
  Clearly, $\Nsp(\gamma'+\rho)\subseteq \Nsp(\gamma+\rho)$ as all the parts in $\theta$ are at least as large as the parts in $\lambda$.
  For the other inclusion, if $(i,j)\in \Nsp(\gamma+\rho)$ then $\lambda_i+\lambda_j\geq \theta_i+\theta_j-1\geq i+j$.
  Thus, we have $(i,j)\in \Nsp(\gamma'+\rho)$ if we establish $\lambda_i+\lambda_j\neq i+j$.
  Proof of Lemma \ref{l:s0case} implies that $\lambda_i+\lambda_j=i+j$ doesn't hold for any choice of $(i, j)$.
  Therefore, we conclude that $\ell(\sigma)=|\Nsp(\gamma+\rho)|=|\Nsp(\gamma'+\rho)|=t$ where the last equality follows from Lemma \ref{l:s0case}.
  Since $q=t+s+u>t$, we conclude that $\ell(\sigma)<q$.

  For the inductive step, we assume that the partition $\theta$ has been obtained form $\lambda \in Q_1(2t)$ by adding $s > 1$ boxes.
  If we remove one of these $s$ boxes from $\theta$ in such a way that the skew partition obtained after the removal remains a partition $\theta'$, then $\theta'$ is obtained from $\lambda$ by adding $s - 1$ boxes.
  It is indeed possible to remove a box from $\theta$ and still get a partition: of the $s$ boxes that were added, we remove the rightmost box with the largest row-index.

  Let $\gamma' = (n - \theta'_n, \dots, 1 - \theta'_1)$.
  Now, by the inductive hypothesis, $|\Nsp(\gamma'+\rho)| < s-1 + t + u$.
  If $k$ is the row of $\theta$ from which the single box was removed to obtain $\theta'$, then the entries of $\theta$ are given by
  \begin{equation*}
    \theta_i =
    \begin{cases}
      \theta'_i +1 &  \text{if } i = k\\
      \theta'_i    &  \text{if } i \neq k
    \end{cases}.
  \end{equation*}
  It is clear from the definition of $\Nsp$ that $|\Nsp(\gamma + \rho)| \geq |\Nsp(\gamma' + \rho)|$.

  Now, we will show that $|\Nsp(\gamma + \rho)| \leq 1 + |\Nsp(\gamma' + \rho)|$.
  We will prove this by contradiction.
  Suppose that $|\Nsp(\gamma + \rho)| - |\Nsp(\gamma' + \rho)| > 1$ after a box is added to row $k$ of $\theta'$.
  This means that there are at least two pairs $(k, j)$ and $(k, l)$ such that $\theta_k + \theta_j > k + j$ and $\theta_k + \theta_{l} > k + l$, but also $\theta'_k+ \theta'_l \leq k + l$ and $\theta'_k + \theta'_j \leq k + j$ as $(k,j), (k,l) \not\in \Nsp(\gamma'+\rho)$.
  The only way this is possible is if $\theta'_k + \theta'_l =  k + l$ and $\theta'_k + \theta'_j = k + j$.
  Since $j \neq l$ are distinct, we may assume $j < l$.
  So we have $\theta'_k + \theta'_l = k + l > k + j = \theta'_k + \theta'_j$.
  This implies that $\theta'_l > \theta'_j$ even though $l > j$ and $\theta'$ is a partition.
  This is a contradiction and we have shown that $|\Nsp(\gamma+\rho)| \leq 1 + |\Nsp(\gamma'+\rho)|$.

  Finally, we use the inductive hypothesis on $|\Nsp(\gamma' + \rho)|$ to conclude
  \begin{equation*}
    \ell(\sigma)
    = |\Nsp(\gamma  +\rho)|
    \leq |\Nsp(\gamma' + \rho)| + 1
    < s - 1 + t + u + 1
    = s + t + u
    = q.
  \end{equation*}
  This completes the proof.
  It is clear from our proof that $\ell(\sigma)=q$ cannot occur when $s>1$.
  When $u>0$ and $s=0$, we have $\ell(\sigma)=t<q$ from Lemma \ref{l:s0case}.
  Thus, equality occurs exactly when we have $t=q$, $s=0$ and $u=0$.
\end{proof}

\begin{lemma}
  \label{l:f0-terms}
  The summand $\bS_\theta \sR \otimes \bS_\zeta \sW$ of $\bigwedge^q \xi$ has nonzero cohomology in degree $q$ exactly when $\theta \in Q_1(2q)$ and $\zeta = \varnothing$.
  Furthermore, in this case, $\rH^q(\IGr(n, L), \bS_\theta\sR) \otimes \bS_\zeta \sW = \bC$.
\end{lemma}

\begin{proof}
  The first part follows from Lemma \ref{l:s>0case}.
  Only thing left to show is that when $\theta\in Q_1(2q)$, $\rH^q(\IGr(n,L),\bS_\theta\sR) = \bC$.
  To see this, observe that from Lemma \ref{l:s0case},  $\rH^q(\IGr(n,L),\bS_\theta\sR)$ is nonzero.
  Let $\sigma$ be the permutation such that $\sigma(\gamma)$ is weakly decreasing with positive entries and let $\alpha=\sigma(\gamma)$.
  To show that the cohomology is the trivial representation, it is enough to check that $\alpha=\rho$.

  As $\rank{\sR} = n$, $l(\theta) \le n$.
  Now $\theta\in Q_1(2q)$ forces $\theta_1\leq n+1$.
  Thus $|\gamma_1|\leq n$ and $|\gamma_n|\leq n$.
  Furthermore, it is easy to see from the definition of $\gamma_i$ that the largest possible value of $|\gamma_i|$ must occur when $i=1$ or $i=n$.
  We conclude that $|\gamma_i|\leq n$ for all $i$.
  Finally, these values must all be distinct and positive, thus forcing $\alpha = \rho$ as required.
  We note that $\alpha$ is just the sequence of $|\gamma_i|$ in decreasing order.
\end{proof}

\begin{remark}
  Lemma~\ref{l:s0case} actually calculates the cohomology of the exterior powers of the cotangent bundle on the isotropic Grassmannian.
  This can be obtained, for instance, using algebraic de Rham cohomology but we reprove it here combinatorially for completeness.
  Some cases of this result appear as exercises in \cite[Chapter~4]{Wey:CohomVect}.
\end{remark}

\subsection{Minimal equations in the two dimensional case}
We now obtain the minimal free resolution of the structure sheaf of the Lagrangian Kalman variety in the case when $\dim L= 2$.
We work with the ambient space $\sp(V)/\Sym^2 W$, see Remark \ref{r:sym2W=0}.

\begin{theorem}
  \label{t:dim-2-min-eq}
  If $\dim{L} = 2$, then the equations in Corollary~\ref{c:defn-eq:lagrangian} give the minimal prime ideal generators of $\IKal(1)$.
\end{theorem}
\begin{proof}
  From Theorem \ref{t:nonnormal}, the nonnormal locus of the Lagrangian Kalman variety is $\IKal(0)$ where
  \begin{equation*}
    \IKal(0)=\{\varphi \in \sp(V) \mid \varphi(W) \subseteq W \}
  \end{equation*}
  is an affine space whose structure sheaf $\O_0$ is resolved by the Koszul complex
  \begin{equation*}
    0
    \to \dots
    \to \bigwedge^2(L\otimes W) \otimes  A(-2)
    \to L \otimes W \otimes A(-1)
    \to A.
  \end{equation*}

  The resolution of the structure sheaf $\widetilde\O_1$ of the normalization of $\IKal(1)$ is obtained using the geometric technique of Theorem~\ref{t:geom-method} where the defining bundle is given by $\xi = \Sym^2(R) \oplus (R \otimes W)$.
  Note that we have
  \begin{equation*}
    \bigwedge^i\xi
    = \left[\Sym^i R \otimes \bigwedge^i W\right]  \oplus  \left[ \Sym^{i+1} R \otimes \bigwedge^{i-1} W \right]
    =
    \begin{array}{c}
      \Sym^i R \otimes \bigwedge^i W \\
      \Sym^{i+1} R \otimes \bigwedge^{i-1} W
    \end{array}.
  \end{equation*}
  To make the terms easier to read, we use the convention that stacks of modules refer to direct sums as shown above.
  When $i>1,$ $\rH^1(\IGr(1, L), \bigwedge^i\xi)$ is nonzero and all other cohomology groups vanish.
  The $(i - 1)$-st syzygy of $\widetilde\O_1$ is thus given by
  \begin{equation*}
    \rH^1(\IGr(1, L), \bigwedge^i\xi) =
    \begin{array}{c}
      \Sym^{i-2} L \otimes \bigwedge^i W \\
      \Sym^{i-1}L\otimes \bigwedge^{i-1}W
    \end{array}.
  \end{equation*}
  When $i = 1$, $\rH^1(\IGr(1, L), \xi) = \bC$.
  Thus, the sequence
  \begin{equation*}
    0
    \to
    \dots
    \to
    \begin{array}{c}
      \Sym^{i-2} L\otimes \bigwedge^i W\\
      \Sym^{i-1}L\otimes \bigwedge^{i-1}W
    \end{array}
    \otimes A(-i)
    \to \dots
    \to
    \begin{array}{c}
      A\\
      A(-1)
    \end{array}
    \to 0
  \end{equation*}
  resolves $\widetilde\O_1$ and, in particular, a presentation of $\widetilde\O_1$ is given by
  \begin{equation*}
    \begin{array}{c}
      \bigwedge^2 W \\
      L \otimes W
    \end{array}
    \otimes A(-2)
    \to \begin{array}{c}
      A\\
      A(-1)
    \end{array} \to 0.
  \end{equation*}
  Since the degree $1$ part of $A$ is $\sp(V)/\Sym^2 W$ and does not contain $\bigwedge^2 W$, the map $\bigwedge^2 W\otimes A(-2)\to A(-1)$ is zero.
  Thus, there is an exact sequence
  \begin{equation*}
    L \otimes W \otimes A(-2) \to A(-1) \to \widetilde\O_1/ \O_1.
  \end{equation*}
  The map $L \otimes W \otimes A(-2) \to A(-1)$ has to be nonzero because otherwise the quotient would be supported everywhere, which is a contradiction as the nonnormal locus is a proper subset.
  Since there is a unique such equivariant nonzero map, we see that $\widetilde\O_1/\O_1\cong \O_0(-1)$.
  In other words, we have an exact sequence
  \begin{equation}
    \label{eq:dim-2-ses}
    0\to \O_1\to \widetilde\O_1 \to \O_0(-1)\to 0.
  \end{equation}
  Using this and the resolutions for $\widetilde\O_1$ and $\O_0(-1)$ we obtain a resolution for $\O_1$ by taking the mapping cone.
  The minimal resolution is obtained after observing that the following cancellations occur in the $i$-th syzygies of the bundles
  \begin{equation}
    \label{eq:tempap}
    \Sym^i L \otimes \bigwedge^i W \otimes A(-i-1)\to \Sym^i L \otimes \bigwedge^i W\otimes A(-i-1).
  \end{equation}
  The fact that these maps between the $i$-th syzygies of $\widetilde\O_1$ and $\O_0(-1)$ are nonzero can be checked inductively.
  We know this holds when $i=1$, so let $i>1$.
  We claim that the map
  \begin{equation}\label{eq:tempap2}
    \Sym^iL \otimes \bigwedge^i W \otimes A(-i-1) \to \Sym^{i-1}L \otimes \bigwedge^{i-1} W \otimes A(-i)
  \end{equation}
  in the minimal free resolution of $\widetilde\O_1$ is nonzero.
  The map
  \begin{equation*}
    \Sym^{i+2} R\otimes \bigwedge^i W\to \Sym^{i+1} R\otimes \bigwedge^{i-1} W\otimes (R\otimes W)
  \end{equation*}
  between bundles in the Koszul complex on $\xi$ is the Koszul map and is therefore injective.
  Let $M$ denote its cokernel.
  We have $M= \Sym^{i+2} R \otimes \bS_{(2,1^{i-2})}W$ and hence using Borel--Weil--Bott we see that $\rH^0(\IGr(1, L), M) = 0$.
  From the long-exact sequence in cohomology,
  \begin{multline*}
    \dots
    \to \rH^0(\IGr(1, L), M)
    \to \rH^1(\IGr(1, L), \Sym^{i+2} R\otimes \bigwedge^i W) \\
    \to \rH^1(\IGr(1, L), \Sym^{i+1} R\otimes \bigwedge^{i-1} W\otimes (R\otimes W))
    \to \dots
  \end{multline*}
  we see that the required map is injective and, in particular, nonzero.

  From the induction hypothesis, the map in \eqref{eq:tempap} is nonzero in the $(i - 1)$-st case.
  Since the map in \eqref{eq:tempap2} is nonzero, this forces the map in \eqref{eq:tempap} in the $i$-th case to be nonzero.
  This completes the induction argument.

  Looking at the first syzygy, we see that the minimal equations of $\IKal(1)$ are given by the representations $\bigwedge^2 W\oplus \Sym^2 W$ where the $\bigwedge^2 W$ part is quadratic and $\Sym^2 W$ part is cubic.
  These representations correspond to the equations coming from $C_0 A_1 = 0$ and ${C_0}{A_0}{A_1}=0$ from Corollary~\ref{c:defn-eq:lagrangian}.
  We conclude that in this case the defining equations are minimal.
\end{proof}

\section{Exact sequence of structure sheaves}
\label{sec:exact-sequence}

In this section, generalizing \eqref{eq:dim-2-ses}, we present a conjecture that implies the existence of a long exact sequence between the structure sheaves of Kalman varieties and their normalizations.
We present some results supporting the validity of the conjecture and outline general heuristics that show the existence of the long exact sequence from the stated conjecture.
The type A analogue of this conjecture was presented in \cite[Conjecture~3.1]{sam2012:equations} and proven in  \cite[Theorem~3.1]{hua2020:equations}.

We are motivated by the observation in Theorem~\ref{t:nonnormal} that
\begin{equation*}
  \nnormal{\IKal(m)} = \IKal(m - 1)
\end{equation*}
and so using the methods in \cite{hua2020:equations}, we can try to relate the structure sheaves of isotropic Kalman varieties and their normalizations to other isotropic Kalman varieties.

In type A, the proof of this result relies on the semisimplicity of the bundle $\xi$ defining the Kalman variety.
The main obstruction in adapting the methods of \cite{hua2020:equations} is the failure of this assumption in type C, see \eqref{eq:iso-cotangent-ses-restr}.

Fix $V$, $L$ and $W$ and denote the structure sheaves and normalizations of $\IKal(m)$ by $\O_{m}$  and $\widetilde\O_m$ respectively.
The resolution of $\widetilde\O_m$ will be denoted $\bF_\bullet^{(m)}$.
If there is ambiguity, we distinguish the defining bundles $\xi$ of $\IKal(m)$ by $\xi^{(m)}$ and the associated Koszul complex on $\IGr(m, L)$ by $\bK^{(m)}_\bullet = \bigwedge^\bullet \xi^{(m)}$.
Now consider the commutative diamond
\begin{equation}
  \begin{tikzcd}[column sep=tiny, row sep=large]
    &
    \IFlag(m - 1, m, L)
    \ar["\pi_1", end anchor=north east]{dl}
    \ar["\pi_2", end anchor=north west, swap]{dr}
    & \\
    \IGr(m, L)
    \ar["p_1", start anchor=south east, swap]{dr}
    & &
    \IGr(m - 1, L)
    \ar["p_1", start anchor=south west]{dl}
    \\
    & \Spec(\bC) &
  \end{tikzcd}.
\end{equation}
Denote the rank $m - 1$ and rank $m$ tautological subbundles over $\IFlag(m - 1, m, L)$ by $\sR_{m - 1}$ and $\sR_m$ respectively; that is, above an isotropic flag $R_{m - 1} \subseteq R_m \subseteq L$, the fibers of $\sR_{m - 1}$ and $\sR$ are isomorphic to $R_{m - 1}$ and $R_m$ respectively. Let $s=2(n-m)+1$.

\begin{conjecture}
  \label{conj:ikal-submod}
  There exists a complex $\bG_\bullet$ of homogeneous vector bundles on $\IFlag(m - 1, m, L)$ such that the pushforward along $\pi_1$ can be identified with a quotient complex
  \begin{equation*}
    \bK_\bullet^{(m)} \surjects (\pi_1)_\ast \bG_\bullet.
  \end{equation*}
  and the pushforward along $\pi_2$ with a subcomplex
  \begin{align*}
    \rR^{s}(\pi_2)_\ast  \bG_\bullet \injects \bK^{(m - 1)}_\bullet[-(n - m)^2].
  \end{align*}
\end{conjecture}

Since terms of $\bF_\bullet^{(m)}$ are obtained by computing cohomology of the terms of $\bK_\bullet^{(m)}$, and we have $p_1\circ \pi_1=p_2\circ \pi_2$, taking pushforwards along the morphisms $p_i$ result in morphisms
\begin{equation*}
  \bF_\bullet^{(m)}
  \surjects \rH^\bullet(\IGr(m, L), (\pi_1)_\ast\bG_\bullet)
  = \rH^\bullet(\IGr(m - 1, L), \rR^{s}(\pi_2)_\ast \bG_\bullet)
  \injects \bF_\bullet^{(m - 1)}[- (n - m)^2]
\end{equation*}
of complexes.
Our main conjecture predicts that these morphisms in fact extend to a long exact sequence.

Given partitions $\lambda'=(\lambda'_1,\dots, \lambda'_{m-1})$ and $\mu=(\mu_1,\dots, \mu_{n-m})$ and a positive integer $k\geq \max(\lambda_1',\mu_1+s+1)$, define new partitions $\lambda = (k, \lambda'_1, \dots, \lambda'_{m - 1})$ and  $\mu' =  (k - s-1, \mu_1, \dots, \mu_{n - m}).$

\begin{conjecture}
  \label{conj:ikal-les}
  There exist maps $\Psi_m \colon \widetilde\O_m \to \widetilde\O_{m - 1}(-s)$ that, for some choices of $\alpha$, map the summands $\rH^\bullet(\IGr(m, L), \bS_\lambda \sR \otimes \bS_{[\mu]}(\R^\perp / \sR)) \otimes \bS_\alpha \sW$ to $\rH^{\bullet}(\IGr(m, L), \bS_{\lambda'} \sR \otimes \bS_{[\mu']}(\R^\perp / \sR)) \otimes \bS_\alpha \sW$ of the resolutions $\bF^{(m)}_\bullet$ and $\bF^{(m-1)}_\bullet.$
  All other summands of $\bF^{(m)}$ are sent to zero under $\Psi_m$.
  Furthermore, the morphisms $\Psi_m$ induce a long exact sequence
  \begin{equation*}
    0 \to
    \O_n \to
    \widetilde\O_n \to
    \widetilde\O_{n - 1}(-1) \to
    \dots \to
    \widetilde\O_m(-(n - m)^2) \to
    \dots \to
    \widetilde\O_0(-n^2) \to
    0.
  \end{equation*}
\end{conjecture}

The description of the morphisms in Conjecture \ref{conj:ikal-les} is motivated by the combinatorics of the homogeneous vector bundles over $\IFlag(m-1,m,L)$ and relative Borel--Weil--Bott. The next proposition elaborates on this.

\begin{proposition}
  \label{p:g-summands}
  Consider the homogeneous bundle $\sG = \sG(\lambda', k, \mu, \alpha)$ on $\IFlag(m - 1, m, L)$ defined by
  \begin{equation*}
    \sG(\lambda', k, \mu, \alpha) =
    \bS_{\lambda'} \sR_{m - 1} \otimes
    (\sR_{m} / \sR_{m - 1})^k \otimes
    \bS_{[\mu]}(\sR_m^\perp / \sR_m) \otimes
    \bS_{\alpha} \sW.
  \end{equation*}

  The following hold for the pushwards of $\sG$ along $\pi_i$:
  \begin{enumerate}
  \item
    Let $\sR$ be the tautological subbundle on $\IGr(m, L)$.
    Then
    \begin{equation*}
      (\pi_1)_\ast  \sG(\lambda', k, \mu, \alpha)
      = \bS_{\lambda} \sR \otimes \bS_{[\mu]} (\sR^\perp / \sR) \otimes \bS_\alpha \sW.
    \end{equation*}

  \item
    Let $\sR$ be the tautological subbundle on $\IGr(m - 1, L)$.
    Then
    \begin{equation*}
      \rR^{s}(\pi_2)_\ast  \sG(\lambda', k, \mu, \alpha)
      = \bS_{\lambda'} \sR \otimes \bS_{[\mu']} (\sR^\perp / \sR) \otimes \bS_\alpha \sW.
    \end{equation*}
  \end{enumerate}

\end{proposition}

\begin{proof}
  First let $\sR$ be the tautological subbundle on $\IGr(m, L)$.
  Since $\sW$ is a trivial vector bundle $(\pi_1)_\ast \bS_\alpha\sW = \bS_\alpha \sW$.
  As $(\pi_1)_\ast (\sR_m^\perp / \sR_m) = (\sR^\perp / \sR)$, $(\pi_1)_\ast \bS_{[\mu]}(\sR_m^\perp / \sR_m) = \bS_{[\mu]}(\sR^\perp / \sR)$.
  For the rest of the terms of $\sG$, identify $\IFlag(m - 1, m, L)$ with the relative Grassmannian $\Gr(m - 1, \sR)$ such that $\pi_1$  plays the role of the structure morphism.
  Under this identification, $\widetilde\sR = \sR_{m - 1}$  and $\widetilde\sQ = \sR_m / \sR_{m - 1}$ where $\widetilde\sR$ and $\widetilde\sQ$ are tautological sub- and quotient bundles on $\Gr(m - 1, \sR)$.
  Since $k \geq \lambda_1'$, using the relative version of Borel--Weil--Bott (Theorem~\ref{t:bwb:rel}), we obtain
  \begin{equation*}
    \rR^i (\pi_1)_\ast \left(\bS_\lambda \sR_{m - 1} \otimes (\sR_m / \sR_{m - 1})^k \right)
    = \rR^i (\pi_1)_\ast\left( \bS_\lambda \widetilde\sR \otimes \bS_{(k)} \widetilde\sQ \right)
    = \bS_{(k, \lambda_1, \dots, \lambda_{m - 1})} \sR
  \end{equation*}
  if and only if $i = 0$.
  So $(\pi_1)_\ast \sG = \bS_{(k, \lambda_1, \lambda_2, \dots, \lambda_{m - 1})} \sR \otimes \bS_{[\mu]} (\sR^\perp / \sR) \otimes \bS_\alpha \sW$ and all higher direct images $\rR^i(\pi_1)_\ast \sG$ vanish for $i \neq 0$.

  For the second claim let $\sR$ denote the tautological subbundle on $\IGr(m - 1, L)$.
  Note that any flag in $\IFlag(m - 1, m, L)$ is determined by picking an $m - 1$ dimensional isotropic subspace $R \subseteq L$ and a one dimensional subspace in $R^\perp / R$. We identify $\IFlag(m - 1, m, L)$ with the relative isotropic Grassmannian $\IGr(1, \sR^\perp / \sR)$ with structure morphism $\pi_2$.
  If $\widetilde\sR$ is the tautological subbundle on $\IGr(1, \sR^\perp / \sR)$, we have identifications $\widetilde\sR = \sR_m / \sR_{m - 1}$ and $\widetilde\sR^\perp / \widetilde\sR = \sR_m^\perp / \sR_m$.
  We are interested in the length \eqref{def:len-type-c} of the signed permutation that sorts
  \begin{align*}
    \gamma + \rho
    &=  (-k, \mu_1, \dots, \mu_{n - m}) + (n - m + 1, \dots, 1) \\
    &=  (n - m + 1 -k, n - m + \mu_1, \dots, 1 + \mu_{n - m})
  \end{align*}
  in descending order.
  As $k \geq \mu_1 + s + 1 = \mu_1 + 2(n - m + 1)$,
  \begin{equation*}
    n - m + 1 - k \leq - \mu_1 - (n - m + 1) \leq 0.
  \end{equation*}

  Since all other entries of $\gamma + \rho$ are nonnegative, $|\Neg(\gamma + \rho)| = 1$.
  Since the last $n - m$ entries of $\gamma + \rho$ form a decreasing sequence of nonnegative integers while the first entry is negative, $|\Inv(\gamma + \rho)| = n - m$.
  Finally, for any $j = 2, \dots, n - m + 1$,
  \begin{equation*}
    (n - m + 1 - k) + (n - m - j + 2 + \mu_{j - 1})
    = [2(n - m + 1) + \mu_{j - 1} - k] + (1 - j)
    < 0
  \end{equation*}
  and so $|\Nsp(\gamma + \rho)| = n - m$.
  We conclude that the signed permutation sorting $\gamma + \rho$ in descending order has length $s = 2(n - m) + 1$.
  Theorem \ref{t:bwb:iso} and Theorem \ref{t:bwb:rel} now imply $i = s$ if and only if
  \begin{multline*}
    \rR^i (\pi_2)_\ast \left( (\sR_m / \sR_{m - 1})^k \otimes \bS_{[\mu]} (\sR_m^\perp / \sR_m) \right) \\
    = \rR^i (\pi_2)_\ast \left(\bS_{(k)} \widetilde\sR \otimes \bS_{[\mu]} (\widetilde\sR^\perp / \widetilde\sR) \right)
    = \bS_{[k - s - 1, \mu_1, \dots, \mu_{n - m}]} (\sR^\perp / \sR).
  \end{multline*}
  Using an argument similar to proof of the first claim on the remaining factors, we see that $\rR^s (\pi_2)_\ast \sG$ is nonzero and identical to the stated bundle.
\end{proof}

\begin{figure}[b]
  \begin{equation*}
    \begin{tikzcd}[column sep=tiny, row sep=small]
  0 \ar{r} & \O_2 \ar{r} &
  \tilde \O_2 \ar["\Psi_2"]{r}
  &
  \tilde \O_1  \ar["\Psi_1"]{r}
  &
  \tilde \O_0  \ar{r}
  &
  0
  \\
  &&
  \boxed{\substack{\Hxi{0}{0}{0;0;0}}}
  \oplus
  {\color{Blue}\substack{\Hxi{1}{1}{2;0;0} \\ \Hxi{2}{2}{3,1;0;0} \\ \Hxi{3}{3}{3,3;0;0}}}
  \ar[Blue]{r}
  \ar{u}
  &
  {\color{Blue}\substack{\Hxi{0}{0}{0;0;0} \\ \Hxi{1}{1}{1;1;0}  \\ \Hxi{2}{2}{3;1;0}}}
  \oplus
  {\color{BrickRed}\substack{\Hxi{3}{3}{4;0;0}}}
  \ar[BrickRed]{r}
  \ar{u}
  &
  {\color{BrickRed}\substack{\Hxi{0}{0}{0;0;0}}}
  \ar{u}
  &
  \\
  &&
  \boxed{\substack{\Hxi{2}{3}{3,1;0;2}}}
  \oplus
  {\color{Blue}\substack{\Hxi{2}{3}{4,1;0;1} \\ \Hxi{3}{4}{4,3;0;1}}}
  \ar{u}
  \ar[Blue]{r}
  &
  {\color{Blue}\substack{\Hxi{1}{2}{1;2;1} \\ \Hxi{2}{3}{3;2;1}}}
  \oplus
  \substack{\Hxi{1}{2}{1;1;2} \\ \Hxi{2}{3}{3;1;2} \\ \Hxi{3}{4}{4;0;2}}
  \oplus
  {\color{BrickRed}\substack{\Hxi{3}{4}{5;0;1}}}
  \ar{u}
  \ar[BrickRed]{r}
  &
  {\color{BrickRed}\substack{\Hxi{0}{1}{0;1;1}}} 
  \ar{u}
  &
  \\
  &&
  {\color{Blue}\substack{\Hxi{3}{5}{4,4;0;2}}}
  \ar{u}
  \ar[Blue]{r}
  &
  {\color{Blue}\substack{\Hxi{2}{4}{4;2;2}}} 
  \oplus
  \substack{
    \substack{\Hxi{0}{2}{0;0;2} \\ \Hxi{1}{3}{1;1;2} \\ \Hxi{2}{4}{3;1;2}} \\
    \boxed{\substack{\Hxi{2}{4}{3;1;2}}}
  }
  \oplus
  {\color{BrickRed}\substack{\Hxi{3}{5}{5;1;2} \\ \Hxi{3}{5}{4;0;2}}}
  \ar[BrickRed]{r}
  \ar{u}
  &
  {\color{BrickRed}\substack{\Hxi{0}{2}{0;1,1;2} \\ \Hxi{0}{2}{0;0;2} }}
  \ar{u}
  \\
  &&
  \substack{0}
  \ar{u}
  &
  {\color{BrickRed}\substack{\Hxi{3}{6}{5;0;3}} \ar{u}}
  \ar[BrickRed]{r}
  &
  {\color{BrickRed}\substack{\Hxi{0}{3}{0;1;3}} \ar{u}}
  \\
  &&
  &
  \substack{0} \ar{u}
  &
  \boxed{\substack{\Hxi{0}{4}{0;0;4}}}
  \ar{u}
  \\
  &&
  &
  &
  \substack{0} \ar{u}
\end{tikzcd}
%%% Local Variables:
%%% mode: LaTeX
%%% TeX-master: "../2025-gkp-isotropic-kalman"
%%% End:
  \end{equation*}
  \caption{\label{fig:dimL4-xi}}
\end{figure}

\begin{example}[$\dim{L} = 4, r = 1$]
  We present some calculations in the case $\dim{L} = 4$, $r = 1$ in support of the conjecture above. We work with the ambient space $\sp(V)/\Sym^2 W$, see Remark \ref{r:sym2W=0}.

  Using the description of the terms of $\bigwedge^\bullet \xi$ given in Lemma~\ref{l:xi-summands} and of the resolution of the normalizations from Proposition~\ref{p:norm-ikal-res}, we obtain the diagram in Figure~\ref{fig:dimL4-xi} where $\Hxi{i}{j}{\lambda;\mu;\alpha}$ denotes the cohomology group $\rH^i(\IGr(m, L), \bS_\lambda \sR \otimes \bS_{[\mu]}(\sR^\perp/ \sR) \otimes \bS_\alpha \sW)$ appearing in $\rH^i(\IGr(m, L), \bigwedge^j \xi)$.
  Now, the morphisms $\Psi_i$ described in Conjecture~\ref{conj:ikal-les} map the rightmost terms in each column match to the leftmost terms of the next column; for example $\Psi_2$ identifies the blue terms and $\Psi_1$ the red terms.
  After computing the cohomology groups using Borel--Weil--Bott (Theorem~\ref{t:bwb:iso}) and tracking the degree shifts, we get the diagram in Figure~\ref{fig:dimL4} where $(\lambda;\alpha)(-k)$ denotes $\bS_{[\lambda]} L \otimes \bS_\alpha W \otimes A(-k)$.

  \begin{figure}[b]
    \begin{equation*}
      \begin{tikzcd}[column sep=tiny, row sep=small]
  0 \ar{r} & \O_2 \ar{r} &
  \tilde \O_2 \ar{r}
  &
  \tilde \O_1 (-1) \ar{r}
  &
  \tilde \O_0 (-4) \ar{r}
  &
  0
  \\
  &&
  \boxed{\substack{A}}
  \oplus
  {\color{Blue}\substack{A(-1) \\ A(-2) \\ A(-3)}}
  \ar[Blue]{r}
  \ar{u}
  &
  {\color{Blue}\substack{A(-1) \\ A(-2)  \\ A(-3)}}
  \oplus
  {\color{BrickRed}\substack{A(-4)}}
  \ar[BrickRed]{r}
  \ar{u}
  &
  {\color{BrickRed}\substack{A(-4)}}
  \ar{u}
  &
  \\
  &&
  \boxed{\substack{\WA{2}{3}}}
  \oplus
  {\color{Blue}\substack{\LWA{1}{1}{3} \\ \LWA{1}{1}{4}}}
  \ar{u}
  \ar[Blue]{r}
  &
  {\color{Blue}\substack{\LWA{1}{1}{3} \\ \LWA{1}{1}{4}}}
  \oplus
  \substack{\WA{2}{3} \\ \WA{2}{4} \\ \WA{2}{5}}
  \oplus
  {\color{BrickRed}\substack{\LWA{1}{1}{5}}}
  \ar{u}
  \ar[BrickRed]{r}
  &
  {\color{BrickRed}\substack{\LWA{1}{1}{5}}}
  \ar{u}
  &
  \\
  &&
  {\color{Blue}\substack{\LWA{1, 1}{2}{5}} \ar{u}}
  \ar[Blue]{r}
  &
  {\color{Blue}\substack{\LWA{1, 1}{2}{5}} }
  \oplus
  \substack{
    \substack{\WA{2}{3} \\ \WA{2}{4} \\ \WA{2}{5}} \\
    \boxed{\substack{\WA{2}{5}}}
  }
  \oplus
  {\color{BrickRed}\substack{\LWA{1, 1}{2}{6} \\ \WA{2}{6} }}
  \ar[BrickRed]{r}
  \ar{u}
  &
  {\color{BrickRed}\substack{\LWA{1, 1}{2}{6} \\ \WA{2}{6} }}
  \ar{u}
  \\
  &&
  \substack{0}
  \ar{u}
  &
  {\color{BrickRed}\substack{\LWA{1}{3}{7}} \ar{u}}
  \ar[BrickRed]{r}
  &
  {\color{BrickRed}\substack{\LWA{1}{3}{7}} \ar{u}}
  \\
  &&
  &
  \substack{0} \ar{u}
  &
  \boxed{\substack{\WA{4}{8}}}
  \ar{u}
  \\
  &&
  &
  &
  \substack{0} \ar{u}
\end{tikzcd}
%%% Local Variables:
%%% mode: LaTeX
%%% TeX-master: "../2025-gkp-isotropic-kalman"
%%% End:
.
    \end{equation*}
    \caption{\label{fig:dimL4}}
  \end{figure}

  Since we are working with the associated graded, the resolution of $\widetilde \sO_1$, as given above, is not minimal.
  Using the defining equations from Theorem~\ref{t:defn-eq} and explicit calculations in Macaulay2 \cite{gs:macaulay2}, we can show that the maps
  \begin{equation*}
    \begin{array}{ccc}
      \WA{0}{2} \\ \WA{0}{3} \\ \WA{0}{4}
    \end{array}
    \to
    \begin{array}{ccc}
      \WA{0}{2} \\ \WA{0}{3} \\ \WA{0}{4}
    \end{array}
  \end{equation*}
  in the resolution of $\widetilde\sO_1$ are given by the identity on each of the components and hence do not appear in the minimal free resolution.
  If $M$ is the submodule of $\widetilde \sO_1$ generated by $A(-1)$, $A(-2)$, $A(-3),$ then the proof in the $\dim{L} = 2$ case can adapted to show the sequence
  \begin{equation*}
    0 \to M \to \widetilde \sO_1 \to \widetilde \sO_0(-3) \to 0
  \end{equation*}
  is exact.
  
  Conjecture~\ref{conj:ikal-les} implies the exactness of the sequence
  \begin{equation*}
    0 \to \sO_2 \to \widetilde \sO_2 \to M(-1) \to 0
  \end{equation*}
  and that the boxed terms contribute to the syzygies of $\sO_2$ resulting in the resolution
  \begin{equation*}
    0\to \Sym^4W \otimes A(-8)\to \Sym^2 W \otimes A(-3)\oplus \Sym^2 W \otimes A(-5) \to A.
  \end{equation*}
  
  We note that the cubics and quintics in the first syzygy correspond to the polynomials $C_0A_0A_1=0$ and $C_0A_0^3A_1=0$ from Corollary~\ref{c:defn-eq:lagrangian}. A simple calculation reveals that $C_0A_1=0$ and $C_0A_0^2A_1=0$ hold (as polynomials). Thus, the conjecture would imply that the set-theoretic equations from Corollary~\ref{c:defn-eq:lagrangian} are in fact the minimal generators of the radical ideal in this case.
\end{example}

\section{Orthogonal Kalman varieties}
\label{sec:ortho}
We now consider the case where $V$ is an orthogonal vector space of dimension $2(n + r) + 1$ (type B) or $2(n + r)$ (type D), that is, $V$ is equipped with a nondegenerate symmetric bilinear form.
In this section, all complements will be with respect to this form and isotropic and coisotropic subspaces will be defined with respect to the orthogonal form.
The orthogonal Grassmannian $\OGr(m, V)$ consists of all isotropic subspaces of dimension $m$; the space $\OGr(m, V)$ is a projective variety but is not irreducible in general.
We will denote the orthogonal Lie algebra by $\so(V)$.

As before, fixing $W \subseteq V$ isotropic of dimension $r$ and $L$ an orthogonal subspace such that $W^\perp = L\oplus W$, we define the orthogonal Kalman variety to be the algebraic subset of $\so(V)$ defined by 
\begin{equation*}
  \OKal(m) = \left\{ \varphi \in \so(V) \;\vert\; \exists U \in \OGr(m+r, V), W\subseteq U, \varphi(U) \subseteq U \right\}.
\end{equation*}
Most of the proofs from the rest of the paper can be adapted, \emph{mutatis mutandis}, to this case using results for type B ($\dim{V}$ odd) or type D ($\dim{V}$ even) Lie algebras in place of the results for type C.
Hence, we only state the main results here and highlight any notable differences from type C.

As in the type A and type C cases, the orthogonal Kalman varieties can be obtained by restricting the cotangent bundle using the framework of \S\ref{sec:restriction-defn}.
If $Y = \OGr(m + r, V)$, then the cotangent bundle $\Omega_Y$ appears in a short exact sequence
\begin{equation}
  \label{eq:ortho-cotangent-ses}
  0 \to \bigwedge^2\R_Y \to \Omega_Y \to \R_Y \otimes (\R_Y^\perp / \R_Y)^\ast \to 0
\end{equation}
analogous to \eqref{eq:iso-cotangent-ses}.
The pullback to $\OGr(m, L)$ along the inclusion then results in an analogue of \eqref{eq:iso-cotangent-ses-restr}.
The resulting set theoretic description follows from an analogue of Proposition~\ref{p:ikal-set-theo}.

\subsection{Odd orthogonal Kalman varieties}
In type B, $\dim{V} = 2(n + r) + 1$, and $\OGr(n + r, V)$ and $\OKal(m)$ are both irreducible.
The set theoretic generators are obtained by the same conditions as those stated in Theorem~\ref{t:defn-eq}.
The resulting polynomials are, however, different as we decompose a generic $\varphi \in \mathfrak{so}(V)$ in terms of an orthogonal basis of $V$ \cite[\S1.2, pp.~2--3]{Hum:IntroductionLie}.
The codimension of $\OKal(m)$ in $\so(V)$ is given by 
\begin{equation*}
  \codim_{\so(V)} \OKal(m) = (2n + 1 - m)r + {r \choose 2}
\end{equation*}
following Proposition~\ref{p:codim}.
As in Theorem~\ref{t:nonnormal}, we have
\begin{equation*}
  \nnormal(\OKal(m)) = \sing(\OKal(m)) = \OKal(m - 1).
\end{equation*}
Finally, conjecturally there is a long exact sequence
\begin{equation*}
  0 \to
  \O_n \to
  \widetilde\O_n \to
  \widetilde\O_{n - 1}(-1) \to
  \dots \to
  \widetilde\O_m(-(n - m)^2) \to
  \dots \to
  \widetilde\O_0(- n^2) \to
  0
\end{equation*}
involving structure sheaves $\O_m$ of $\OKal(m)$ and their normalizations $\widetilde \O_m$ (cf. Conjecture~\ref{conj:ikal-les}).

\subsection{Even orthogonal Kalman varieties}
In type $D$, $\dim{V} = 2(n + r)$.
When $m < n$, $\OKal(m)$  is irreducible, $\codim_{\so(V)} \OKal(m) = (2n + 1 - m) r + {r \choose 2}$, and the nonnormal and singular loci coincide and we have $\nnormal(\OKal(m)) = \OKal(m - 1)$.
The defining equations in this case are analogous to the odd orthogonal case.

If $m = n$, $\OGr(n + r, V)$ is reducible and has two irreducible components $\OGr^+(n + r, V)$ and $\OGr^-(n + r, V)$ that are isomorphic to each other.
As a consequence, $\OKal(n)$ also has two isomorphic irreducible components $\OKal^\pm(n)$ whose intersection is $\OKal(n - 2)$.
For each of the irreducible components,
\begin{equation*}
  \codim_{\so(V)} \OKal^\pm(n) = nr + {r \choose 2}.
\end{equation*}
The set theoretic equations from Theorem~\ref{t:defn-eq} define $\OKal(m)$; the additional equations describing the irreducible components $\OKal^\pm(m)$ can be obtained by adapting \cite[Proposition~5.8]{lor2023:singularities}.

The singular and normal loci coincide and we have
\begin{equation*}
  \nnormal(\OKal(n)) = \sing(\OKal(n)) = \OKal(n - 2).
\end{equation*}
Thus, in type D, the conjectural long exact sequence resembles
\begin{multline*}
  0 \to
  \O_n \to
  \widetilde\O_n \to
  \widetilde\O_{n - 2}(-1) \to
  \dots \\ \dots \to
  \widetilde\O_m(-(n - m)(n-m-1)+1) \to
  \dots \to
  \widetilde\O_0(- n(n-1)+1) \to
  0
\end{multline*}
where $\O_m$ and $\widetilde \O_m$ have the usual meaning.

\bibliographystyle{alphaurl}
\bibliography{isotropic-kalman-refs}
\end{document}